\documentclass[12pt]{amsart}
\usepackage{amscd,amssymb,latexsym}
\usepackage{fullpage}
\newcommand{\Mdef}[2]{\newcommand{#1}{\relax \ifmmode #2 \else $#2$\fi}}

\newcommand{\im}{\mathrm{im}}

\newcommand{\tensor}{\otimes}

\newcommand{\Hom}{\mathrm{Hom}}

\Mdef{\bhom}{\mathbf{\hat{H}om}}

\Mdef{\Mod}{\mathrm{mod}}

\newcommand{\st}{\; | \;}

\newtheorem{thm}{Theorem}[section]
\newtheorem{lemma}[thm]{Lemma}
\newtheorem{prop}[thm]{Proposition}
\newtheorem{cor}[thm]{Corollary}

\theoremstyle{definition}

\newtheorem{defn}[thm]{Definition}

\newtheorem{example}[thm]{Example}

\newtheorem{remark}[thm]{Remark}

\newcommand{\qqed}{\qed \\[1ex]}
\renewenvironment{proof}[1][\hspace*{-.8ex}]{\noindent {\bf Proof #1:\;}}{\qqed}

\Mdef{\PH} {\Phi^H}
\Mdef{\PK} {\Phi^K}
\Mdef{\PL} {\Phi^L}
\Mdef{\PT} {\Phi^{\T}}

\Mdef{\ef}{E{\cF}_+}
\Mdef{\etf}{\widetilde{E}{\cF}}
\Mdef{\eg}{E{G}_+}
\Mdef{\etg}{\tilde{E}{G}}

\Mdef{\infl}{\mathrm{inf}}
\Mdef{\defl}{\mathrm{def}}
\Mdef{\res}{\mathrm{res}}
\Mdef{\ind}{\mathrm{ind}}
\Mdef{\coind}{\mathrm{coind}}

\Mdef{\univ}{\mathcal{U}}

\Mdef{\Fp}{\mathbb{F}_p}
\Mdef{\Zpinfty}{\Z /p^{\infty}}
\Mdef{\Zpadic}{\Z_p^{\wedge}}

\newcommand{\bi}{\begin{itemize}}
\newcommand{\be}{\begin{enumerate}}
\newcommand{\bc}{\begin{center}}
\newcommand{\bd}{\begin{description}}
\newcommand{\ei}{\end{itemize}}
\newcommand{\ee}{\end{enumerate}}
\newcommand{\ec}{\end{center}}
\newcommand{\ed}{\end{description}}
\newcommand{\dichotomy}[2]{\left\{ \begin{array}{ll}#1\\#2 \end{array}\right.}

\newcommand{\adjunction}[4]{
\diagram
#1:#2 \rrto<0.7ex> &&
#3  \llto<0.7ex> :#4 
\enddiagram}
\newcommand{\lra}{\longrightarrow}
\newcommand{\lla}{\longleftarrow}

\newcommand{\sets}{\mathbf{Sets}}
\newcommand{\crings}{\mathbf{Rings}}

\newcommand{\rings}{\mathbf{Rings}}

\Mdef{\we}{\mathbf{we}}
\Mdef{\fib}{\mathbf{fib}}
\Mdef{\cof}{\mathbf{cof}}
\Mdef{\BI}{\mathcal{BI}}

\newcommand{\ilim}{\mathop{ \mathop{\mathrm{lim}} \limits_\leftarrow} \nolimits}
\newcommand{\colim}{\mathop{  \mathop{\mathrm {lim}} \limits_\rightarrow} \nolimits}

\Mdef{\A}{\mathbb{A}}
\Mdef{\B}{\mathbb{B}}
\Mdef{\C}{\mathbb{C}}
\Mdef{\D}{\mathbb{D}}
\Mdef{\E}{\mathbb{E}}
\Mdef{\T}{\mathbb{T}}
\Mdef{\F}{\mathbb{F}}
\Mdef{\G}{\mathbb{G}}
\Mdef{\I}{\mathbb{I}}
\Mdef{\N}{\mathbb{N}}
\Mdef{\Q}{\mathbb{Q}}
\Mdef{\R}{\mathbb{R}}
\Mdef{\bbS}{\mathbb{S}}
\Mdef{\Z}{\mathbb{Z}}

\Mdef{\bA}{\mathbb{A}}
\Mdef{\bB}{\mathbb{B}}
\Mdef{\bC}{\mathbb{C}}
\Mdef{\bD}{\mathbb{D}}
\Mdef{\bE}{\mathbb{E}}
\Mdef{\bF}{\mathbb{F}}
\Mdef{\bG}{\mathbb{G}}
\Mdef{\bH}{\mathbb{H}}
\Mdef{\bI}{\mathbb{I}}
\Mdef{\bJ}{\mathbb{J}}
\Mdef{\bK}{\mathbb{K}}
\Mdef{\bL}{\mathbb{L}}
\Mdef{\bM}{\mathbb{M}}
\Mdef{\bN}{\mathbb{N}}
\Mdef{\bO}{\mathbb{O}}
\Mdef{\bP}{\mathbb{P}}
\Mdef{\bQ}{\mathbb{Q}}
\Mdef{\bR}{\mathbb{R}}
\Mdef{\bS}{\mathbb{S}}
\Mdef{\bT}{\mathbb{T}}
\Mdef{\bU}{\mathbb{U}}
\Mdef{\bV}{\mathbb{V}}
\Mdef{\bW}{\mathbb{W}}
\Mdef{\bX}{\mathbb{X}}
\Mdef{\bY}{\mathbb{Y}}
\Mdef{\bZ}{\mathbb{Z}}

\Mdef{\cA}{\mathcal{A}}
\Mdef{\cB}{\mathcal{B}}
\Mdef{\cC}{\mathcal{C}}
\Mdef{\mcD}{\mathcal{D}} 
\Mdef{\cE}{\mathcal{E}}
\Mdef{\cF}{\mathcal{F}}
\Mdef{\cG}{\mathcal{G}}
\Mdef{\mcH}{\mathcal{H}} 
\Mdef{\cI}{\mathcal{I}}
\Mdef{\cJ}{\mathcal{J}}
\Mdef{\cK}{\mathcal{K}}
\Mdef{\mcL}{\mathcal{L}}

\Mdef{\cM}{\mathcal{M}}
\Mdef{\cN}{\mathcal{N}}
\Mdef{\cO}{\mathcal{O}}
\Mdef{\cP}{\mathcal{P}}
\Mdef{\cQ}{\mathcal{Q}}
\Mdef{\mcR}{\mathcal{R}}
\Mdef{\cS}{\mathcal{S}}
\Mdef{\cT}{\mathcal{T}}
\Mdef{\cU}{\mathcal{U}}
\Mdef{\cV}{\mathcal{V}}
\Mdef{\cW}{\mathcal{W}}
\Mdef{\cX}{\mathcal{X}}
\Mdef{\cY}{\mathcal{Y}}
\Mdef{\cZ}{\mathcal{Z}}

\Mdef{\ca}{\mathcal{a}}
\Mdef{\ct}{\mathcal{t}}

\Mdef{\At}{\tilde{A}}
\Mdef{\Bt}{\tilde{B}}
\Mdef{\Ct}{\tilde{C}}
\Mdef{\Et}{\tilde{E}}
\Mdef{\Ht}{\tilde{H}}
\Mdef{\Kt}{\tilde{K}}
\Mdef{\Lt}{\tilde{L}}
\Mdef{\Mt}{\tilde{M}}
\Mdef{\Nt}{\tilde{N}}
\Mdef{\Pt}{\tilde{P}}

\Mdef{\tA}{\tilde{A}}
\Mdef{\tB}{\tilde{B}}
\Mdef{\tC}{\tilde{C}}
\Mdef{\tE}{\tilde{E}}
\Mdef{\tH}{\tilde{H}}
\Mdef{\tK}{\tilde{K}}
\Mdef{\tL}{\tilde{L}}
\Mdef{\tM}{\tilde{M}}
\Mdef{\tN}{\tilde{N}}
\Mdef{\tP}{\tilde{P}}

\Mdef{\ft}{\tilde{f}}
\Mdef{\xt}{\tilde{x}}
\Mdef{\yt}{\tilde{y}}

\Mdef{\Ab}{\overline{A}}
\Mdef{\Bb}{\overline{B}}
\Mdef{\Cb}{\overline{C}}
\Mdef{\Db}{\overline{D}}
\Mdef{\Eb}{\overline{E}}
\Mdef{\Fb}{\overline{F}}
\Mdef{\Gb}{\overline{G}}
\Mdef{\Hb}{\overline{H}}
\Mdef{\Ib}{\overline{I}}
\Mdef{\Jb}{\overline{J}}
\Mdef{\Kb}{\overline{K}}
\Mdef{\Lb}{\overline{L}}
\Mdef{\Mb}{\overline{M}}
\Mdef{\Nb}{\overline{N}}
\Mdef{\Ob}{\overline{O}}
\Mdef{\Pb}{\overline{P}}
\Mdef{\Qb}{\overline{Q}}
\Mdef{\Rb}{\overline{R}}
\Mdef{\Sb}{\overline{S}}
\Mdef{\Tb}{\overline{T}}
\Mdef{\Ub}{\overline{U}}
\Mdef{\Vb}{\overline{V}}
\Mdef{\Wb}{\overline{W}}
\Mdef{\Xb}{\overline{X}}
\Mdef{\Yb}{\overline{Y}}
\Mdef{\Zb}{\overline{Z}}

\Mdef{\db}{\overline{d}}
\Mdef{\hb}{\overline{h}}
\Mdef{\qb}{\overline{q}}
\Mdef{\rb}{\overline{r}}
\Mdef{\tb}{\overline{t}}
\Mdef{\ub}{\overline{u}}
\Mdef{\vb}{\overline{v}}

\Mdef{\hc}{\hat{c}}
\Mdef{\he}{\hat{e}}
\Mdef{\hf}{\hat{f}}
\Mdef{\hA}{\hat{A}}
\Mdef{\hH}{\hat{H}}
\Mdef{\hJ}{\hat{J}}
\Mdef{\hM}{\hat{M}}
\Mdef{\hP}{\hat{P}}
\Mdef{\hQ}{\hat{Q}}

\Mdef{\thetab}{\overline{\theta}}
\Mdef{\phib}{\overline{\phi}}

\Mdef{\uA}{\underline{A}}
\Mdef{\uB}{\underline{B}}
\Mdef{\uC}{\underline{C}}
\Mdef{\uD}{\underline{D}}

\Mdef{\bolda}{\mathbf{a}}
\Mdef{\boldb}{\mathbf{b}}
\Mdef{\bfD}{\mathbf{D}}

\Mdef{\fm}{\frak{m}}
\Mdef{\fp}{\frak{p}}

\Mdef{\eps}{\epsilon}

\input{xypic}
\newcommand{\connsub}{\mathrm{ConnSub}}

\newcommand{\flag}{\mbox{flag}}
\newcommand{\cEi}{\cE^{-1}}
\newcommand{\mbd}{\mathbf{d}}
\newcommand{\sub}{\mathrm{Sub}}
\newcommand{\ist}{i_{\sigma}^{\tau}}

\newcommand{\siftyV}[1]{S^{\infty V(#1)}}

\newcommand{\Sigmat}{\widetilde{\Sigma}}
\newcommand{\Sigmab}{\overline{\Sigma}}
\newcommand{\qp}{P}
\newcommand{\pes}{\pi^e_{!}}
\newcommand{\RR}{{\mathbb{R}}}
\newcommand{\RRa}{{\mathbb{R}}_a}
\newcommand{\RRc}{{\mathbb{R}}_c}
\newcommand{\RRcb}{{\overline{\mathbb{R}}}_c}
\newcommand{\RRd}{{\mathbb{R}}_d}

\newcommand{\RRb}{\overline{R}}
\newcommand{\RRf}{R^f}
\newcommand{\RRbf}{\RRb^f}
\newcommand{\TC}{\mathcal{TC}}

\newcommand{\cAp}{\cA^p}
\newcommand{\cAf}{\cA^f}

\begin{document}
\title{Rational torus-equivariant stable homotopy III: comparison of
   models.}
\author{J.P.C.Greenlees}
\address{School of Mathematics and Statistics, Hicks Building, 
Sheffield S3 7RH. UK.}
\email{j.greenlees@sheffield.ac.uk}
\date{}

\begin{abstract}
We give details of models for rational torus equivariant homotopy theory
(a) based  on   all subgroups, connected subgroups or dimensions of
subgroups and (b) based on pairs of subgroups or general flags of subgroups. We provide comparison
functors and show the models are equivalent. 
\end{abstract}

\thanks{I am grateful to MSRI for support and providing an excellent environment
  for organizing these ideas during the Algebraic Topology Programme
  in 2014}
\maketitle

\tableofcontents

\section{Background}
\label{sec:motivation}

The author's long standing project aims to give algebraic models for
rational stable equivariant homotopy categories. More precisely, the
aim is to establish a Quillen equivalence between the category of
rational $G$-spectra and the category of differential graded objects
in an abelian category $\cA (G)$. 

The most complicated results to date have been in the case when $G$ is a torus. 
As the project has developed, the technical details of different parts
of the argument have made
it convenient to formulate the definition of the category $\cA (G)$ in
several different ways. The narrow purpose of this paper
is to give a systematic explanation of why these different
formulations give equivalent categories.

The very purpose of this paper is to provide the proper language to describe the
comparison, so we cannot fully describe our results 
until we have introduced a considerable amount of
infrastructure. For the purposes of the introduction we content
ourselves with an informal account.

\subsection{History}
\label{subsec:history}
The category  $\cA (G)$ is assembled from
isotropical  information associated to the various closed subgroups of $G$,  and it is principally
the way this information is put  together that has evolved. 
 
 When $G$ is the circle  group,  the models are simple enough for
 comparisons to be side comments in \cite{s1q}, but already there is a distinction between
 whether the data is indexed by {\em connected} closed subgroups ($c$) or  
{\em all} closed subgroups ($a$). Only the $c$ version was made explicit in \cite{s1q},
despite numerous motivations implicitly using the $a$ version.  

When it comes to higher tori, the distinction between $a$ and $c$
versions is just as important, but for simplicity we just discuss the
$c$ version for now. 
The paper  \cite{tnq1} gives an account of a model $\cA
(G)$. However it became important in \cite{tnqcore} to make some of the
details more explicit and a variant was introduced in
\cite{tnq2}. In the present paper,  this version from \cite{tnq2} is called
$\cA^p_c(G)$, where  $p$ refers to the fact that the data is assembled
from {\em pairs} of subgroups. The
category $\cAp_c(G)$  is the algebraically simplest formulation, and will remain the essential
basis for calculation. 

On the other hand, the proof in \cite{tnqcore}
is homotopical, so it is essential to explicitly include information
about how pairs of subgroups fit together to make longer chains, and it is
convenient to treat all subgroups of the same dimension
together. Altogether we therefore have a structure based on dimensions
of flags, and it is shown that the sphere spectrum is a homotopy
pullback of a diagram of ring spectra indexed on this diagram. This
leads to an algebraic model $\cA_d^f(G)$ based on {\em dimensions} of
{\em flags} of 
subgroups, and this model is essential for comparison with homotopy. 

Our narrow purpose is therefore to show explicitly that  $\cAp_c(G)$ is
equivalent to $\cA_d^f(G)$. 

\subsection{Going further}

We have described the motivation explicit in past treatments, but
the present paper also completes the unfinished business of relating the $a$ and $c$
versions.  It has always been clear that the  data in the $c$ model
is assembled from that in the $a$ model, and that the data in the $a$
model can be recovered from the $c$ model. We identify here precisely what additional
structure is required in the $a$ model to give an equivalence between
the models. It turns out that the additional data should be
thought of as a continuity condition on the subgroups with a given 
identity component; the continuity condition in the $a$ models is
reflected purely algebraically in the $c$ models, and both languages are useful. 

There are at least four further benefits from understanding this type
of algebraic model and for being able to move routinely between them. 

Firstly, it is becoming clear  that when groups other than the torus are
concerned, the partially ordered set structures considered here will
need to be augmented in two ways (a)  by adding an 
action of a finite group, just as representation theory works with the
action of the Weyl group on the maximal torus (see
\cite{AGtoral}) and (b)  the diagrams of rings and
modules considered here can be viewed as sheaves over {\em discrete} topological
categories; in the more general case we should consider sheaves over
the topological category whose {\em space} of objects consists of the
closed subgroups  with
a   topology taking into
account proximity for subgroups of finite index in their normalizers 
(see \cite{ratmack, o2q}).  It is essential to have a flexible formal
framework as described in the present paper; this has already been useful in the
construction of  the toral part of the model for an arbitrary
group \cite{AGtoral}. 

Secondly,  the algebra of these models needs to be well understood to
construct and 
work with the Quillen model structures on the categories of differential
graded modules. Thirdly, a detailed understanding of the models is
extremely valuable when dealing with the additional subtleties of modelling change of groups.

Finally, the structures we introduce here will be described in terms of 
diagrams of rings and modules over them, and these same diagrams 
can be viewed as descriptions of rings of functions on algebraic
varieties and subvarieties. At the crudest level, these give an
algebraic geometric view of what we are doing, but more significantly it lets us
sysematically construct objects in $\cA (G)$ from purely
geometric data as in \cite{ellT, derellT} and \cite{cgq}. We plan to return to this
in \cite{flagsGA}.

\subsection{Increasing precision}

As alluded to in Subsection \ref{subsec:history} above, the algebraic
model $\cA (G)$ is based on considering categories of modules over 
diagrams of rings. The diagrams will be rather simple in the sense
that they are functors $\RR : \Sigma \lra \crings$ which have the shape
of a poset $\Sigma$. Although several different posets will be
involved, the functor $\RR$ will be essentially the same throughout.
     The three variables are 
\begin{description}
\item[The poset $\Sigma$ describing the shape of the diagrams] this is
  indicated  by a letter from 
$\{ a,c,d\}$,  corresponding to  {\bf
   a}ll closed subgroups, {\bf c}onnected subgroups or {\bf
   d}imensions of subgroups.
\item[The type of diagram we build from $\Sigma$] this is indicated by
  by a letter  from $\{ s, p, f \}$, corresponding to whether it is  based on
{\bf s}ingle subgroups, {\bf p}airs of subgroups,   or {\bf   f}lags of subgroups. 
\item[The conditions placed on the modules] Here this means a
 binary choice for each of $\{ qc, e, p, cts\}$ namely quasicoherence ($qc$),
 extendedness ($e$),  product decompositions ($p$) or continuity
 ($cts$). 
\end{description}

Not all of these $3\times 3\times 2^4$ combinations are relevant, but
even for a single group $G$ it is clear we need a systematic framework
for discussing them.  The  categories are connected by a web of adjoint pairs of
functors. Since the functor $\RR$ is the same throughout we will indicate the
domain by a subscript ($a,c$ or $d$) and the type of poset by a
superscript ($s,p$ or $f$).  Omitting the conditions on the modules which may be
necessary to define the functors, we will describe adjoint pairs
$$\diagram
\mbox{$\RRa^{p}$-modules} \rto<0.7ex> &
\mbox{$\RRc^{p}$-modules} \rto<0.7ex> \lto<0.7ex>&
\mbox{$\RRc^{f}$-modules} \rto<0.7ex> \lto<0.7ex>&
\mbox{$\RRd^{f}$-modules} \lto<0.7ex>.\\
\enddiagram$$

The point is that the pullback square of ring spectra from \cite{tnqcore}
delivers a coefficient system $\RRd^f$ on the punctured cube of
non-empty subsets of $\{0,1, \ldots , r\}$ so we are committed to
the use of  the right hand end. On the other hand, the essential part of
the structure is the localization theorem, which (when $G$ is a torus) 
delivers a diagram based on pairs of subgroups. The idempotents of the Burnside ring
allow us to separate subgroups with the same identity component, so 
we may represent that information in $\RRc^p$-modules or
$\RRa^p$-modules at the left hand end. 

\subsection{The plan}
For the rest of the paper we will steadily introduce language to give
a general treatment of the structures that concern us. The particular
examples from this motivational section will be introduced properly at the
appropriate point in the discussion. Once the machinery is introduced, 
in Section \ref{sec:AG}  we give details of the comparison of the models of
    rational $G$-spectra, with some results about torsion functors
    adapted from \cite{tnq2} deferred to Section \ref{sec:Gamma}.

\section{Splitting systems and Euler classes}
\subsection{Flags}
We suppose given a countable partially ordered set $\Sigma$. The prime example
is that $\Sigma$ consists of the connected subgroups of a torus $G$,
and the notation is chosen accordingly. The order relation is
$G\supseteq H \supseteq K \supseteq L$ with $G$ denoting the top 
element. We do not want to insist on a bottom element. 
The maximal elements (i.e., $H$ so that $H'\supset H$ implies $H'=G$) will play a special role.  

The motivating examples are as follows. 
\begin{example}
\label{eg:sigma}
(i) The partially ordered set $\Sigma_c=\connsub (G)$ of connected subgroups of a torus $G$
under containment. 

(ii) The poset $\sub (G)$ of all closed subgroups of a compact Lie
group $G$, under containment. In fact this example will not be very
relevant to us, but a certain non-full poset will be. 

(iii) The poset $\Sigma_a=\TC (G)$ of all closed subgroups of a compact Lie group $G$, with
$L \subseteq  K$ if $L$ is normal in $K$ with a torus quotient \cite{ratmack}. We emphasize
that this has many fewer morphisms than the poset with all
inclusions. In the applications it is this toral-chain poset that is
relevant. 

(iv) The set $\Sigma_d=[0,r]:= \{ i \in \Z \st 0\leq i \leq r\}$ with the
usual ordering of integers. 
\end{example}

A sequence of elements 
$$F=(H_0 \supset H_1 \supset \cdots \supset H_s)$$
is called an {\em $s$-flag}, and we write $|F|=s$. We call $H_0=f(F)$ the {\em first} element of $F$ and  $H_s=l(F)$ the
{\em last} element of $F$. It is worth emphasizing that the biggest
element of the flag is first (this is to take notational advantage of standard
bracketing conventions in one of the applications). 

We write
$$\flag_s (\Sigma) =\{ F \st |F|=s\}, $$
and 
$$\flag (\Sigma)=\bigcup_s \flag_s(\Sigma). $$  

We note that for $s\geq 1$ we have  maps $\partial_i: \flag_s(\Sigma)\lra
\flag_{s-1} (\Sigma)$ for $i=0, \ldots , s$ by omitting the $i$th
term.  If we permitted degenerate flags (i.e., containing equalities)  we
would obtain a simplicial set, but instead we simply view $\flag
(\Sigma)$ as poset. 

Finally, we will need to consider various maps of  posets, such as the dimension function
$d: \Sigma \lra I= [0,r]=\{ 0, 1, \ldots , r\}$  with $d(H)=\dim (H)$. 

\subsection{$\Sigma$-diagrams of rings}

\begin{defn} 
(i) A {\em $\Sigma$-splitting diagram} is  a diagram $R: \Sigma^{op}\lra \rings$ of
rings. We may write $R(G/L)$ for the value  at $L$. If
$K\supseteq L$ we write $\infl_{G/K}^{G/L}: R(G/K)\lra R(G/L)$ and call it
{\em inflation} from $G/K$ to $G/L$.

(ii) A {\em system of Euler classes} for a splitting diagram $R$ is a
collection of functors $\cE_{/L}: \Sigma_{\supseteq L}\lra Mult
(R(G/L))$ from elements above $L$ to multiplicatively closed
subsets of $R(G/L)$; the functoriality is the statement that
$\cE_{K/K}=\{ 1\}$ and that  if $L\subseteq K \subseteq H$ then 
$$\cE_{K/L}\subseteq \cE_{H/L}.$$
 These functors are said to be a {\em transitive system} if whenever
 $H\supseteq K \supseteq L$ the multiplicative system $\cE_{H/L}$ is
 generated by $\cE_{K/L}$ and the inflation of the one for $\cE_{H/K}$:
$$\cE_{H/L}=\langle \infl_{G/K}^{G/L}\cE_{H/K} ,\cE_{K/L} \rangle $$
\end{defn}

\begin{remark}
The notation $G/K$ has no meaning in itself. However $R(G/K)$ is 
supposed to suggest that information for objects above $K$ is being captured. 
\end{remark}

\begin{defn}
The  systems of Euler classes we are most concerned with will all be 
 {\em maximally generated } in the following sense. 
For each maximal element $H$ in $\Sigma$ we are given elements 
$e^H_i\in R (G/H)$ for $i \in I(H)$. We then obtain a transitive 
system by defining
$$\cE_{K/L}=\langle \infl_{G/H}^{G/L} (e^H_i)\st L\subseteq H, K \not \subseteq H\rangle .$$
In our main example there is just one element $e^H \in R (G/H)$ for
each maximal $H$. 
\end{defn}

\begin{example}
\label{eg:BorelTorus}
Given a torus $G$ we may let $\Sigma =\sub (G)$ 
and take $d(H)=\dim (H)$.  The most important splitting diagram for us
is the diagram $\RR$ of polynomial rings defined by 
$$\RR (G/L):=H^*(BG/L; \Q). $$
This is the diagram referred to in Section \ref{sec:motivation}, and
the notation $\RR$ will be reserved for this use throughout. 

More generally, given a cohomology theory $E$, we obtain an splitting
diagram $ \bE$ by taking 
$$\bE (G/L):=E^*(BG/L),  $$
and where inflation has its usual meaning. The main example $\RR$ is
the one corresponding to rational ordinary cohomology: $\RR =\mathbb{HQ}$. 

If in addition $E$ is complex orientable, Euler classes $e_G(V)$ are
defined for complex representations $V$ of $G$. These are compatible
with inflation in the sense that if $W$ is a representation of $G/L$
then $e_G(\infl_{G/L}^{G/1} W)=\infl_{G/L}^{G/1} e_{G/L}(W)$, so we may omit the subscript $G$
without confusion. 

Now take
$$\cE_{K/L}=\{ e(V)\st V \mbox{ is a representation of $G/L$ with }
V^{K/L}=0\}. $$
This evidently gives a system of Euler classes and it is  transitive
since
$$e(V)=e(V^{K/L}) e(V/V^{K/L}).$$

Now we note that we have inclusions $\connsub (G) \subset \TC
(G)\subset  \sub (G)$. The poset $\sub (G)$ does not have maximal
elements, but the maximal elements in both $\TC (G)$ and $\connsub
(G)$ are the codimension $1$ subgroups they contain. 

For $\Sigma =\TC (G)$ or for $\Sigma =\connsub (G)$, this example is 
maximally generated; for each maximal subgroup $H$ we
choose one of the two faithful representations of $G/H$ and call it
$\hat{H}$. The  system of Euler classes is maximally generated by
$$e(\hat{H})\in E^*(BG/H)=E^* [[e(\hat{H})]]. $$ 
\end{example}

\begin{example}
As a slight generalization of Example \ref{eg:BorelTorus}, we may suppose 
given any global equivariant theory $E$, and define
$$\bE (G/L) =E^*_{G/L}.  $$
If the cohomology theory is globally complex stable (i.e., all equivariant theories are
complex stable, and the Thom isomorphisms are compatible with
inflation), we define a  system of  Euler
classes as before. Again this is maximally generated by the Euler
classes $e(\hat{H})\in E^*_{G/H} $.

This cohomological example explains the terminology, since one 
says that a cohomology theory is split if there is a ring map
$\infl_{G/G}^{G/1} E \lra E$ which is a non-equivariant
equivalence. Taking fixed points we obtain a map $E^*\lra E_G^*$. For
global  equivariant theories, we have ring maps
$\infl_{G/K}^{G/1}(E_{G/K})\lra E_G$, showing they are split. 
\end{example}


\subsection{Coefficient systems on the flag complex}

Given a splitting system $R$ on $\Sigma$ with Euler classes there is
an associated {\em coefficient system} on  $\flag
(\Sigma) $. When helpful for emphasis, we write $R^s$ for the original
splitting system and $R^f$ for the associated coefficient system.  

 First we note that $\flag (\Sigma)$ is itself a poset
where $E\leq F$ if $E$ is obtained by omitting some terms from $F$. 
We may define the (covariant) functor
$$R^f: \flag (\Sigma )\lra \rings$$
by
$$R^f(F):=\cEi_{H_0/H_1}\cEi_{H_1/H_2}\cdots
\cEi_{H_{s-1}/H_s}R(G/H_s)=\cEi_{H_0/H_s} R(G/H_s),$$
where the second equality uses the fact that the system
is transitive.

We note that $R^f(F)$ is {\em middle-independent} in the sense that the
values are unaffected by omitting middle vertices: 
$$R^f(\partial_iF)=R^f(F) \mbox{ if } 0<i<s. $$
On the other hand, omitting the first element we have a localization map   
$$R^f(\partial_0F)=\cEi_{H_1/H_s} R(G/H_s) \lra \cEi_{H_0/H_s} R(G/H_s)
=R^f(F),$$
and omitting the last element we have an inflation map
$$R^f(\partial_sF)=\cEi_{H_0/H_{s-1}} R(G/H_{s-1}) \lra \cEi_{H_0/H_s}
R(G/H_s) =R^f(F). $$

\begin{remark}
\label{rem:phinotflag}
The splitting system $R^s$ should not be confused with the
coefficient system $R^f$. The notational distinction between $R(G/H)$
(value of  the splitting system at $H$) and $R(H)$ (value of the
coefficient system at the flag $H$ of length 0) should help. 

The point to bear in mind is 
that the coefficient system $R^f$ includes the {\em values} of the splitting system
 as the values on length 0 flags: $R^f(K)=R^s(G/K)$. However the {\em maps} of
 the splitting system  are not included.
If $H\supset K$ there is an inflation map $R(G/H) \lra R(G/K)$, but
in $\flag (\Sigma)$ there is no direct relation between the flag $(H)$ 
and the flag $(K)$. The flag $(H\supset K)$ gives inclusions
$$(H)\lra (H\supset K)\lla (K)$$ and hence ring maps
$$R^f(H)=R^f(\partial_1(H\supset K))\lra R^f(H\supset K)\lla R^f(\partial_0(H\supset K))=R^f(K). $$
In our case these become
$$R(G/H) \lra \cEi_{H/K}R(G/K)\lla R(G/K). $$
\end{remark}

\subsection{Modules over the coefficient system}

Note that the coefficient system  $R$ is a $\flag (\Sigma)$-diagram of
rings, and we may consider modules over it. Explicitly, $M(F)$ is an
$R(F)$-module, and if $E\leq F$ there is a map $M(E)\lra M(F)$ over
the structure map $R(E)\lra R(F)$. 

\begin{defn}
We say that $M$ is a {\em qce-module} if, for all inclusions
$E\subseteq F$, the map  $M(E) \lra M(F)$ induces
an isomorphism 
$$R(F)\tensor_{R(E)}M(E)\stackrel{\cong}\lra M(F). $$
\end{defn}

\begin{remark}
By associativity of the tensor product, the value of a
$qce$-module is
determined by the structure maps and the values on length 0 flags. 

In particular, it is {\em last-determined} in the sense that  for any flag
$F=(L_0\supset \cdots  \supset L_t)$,  we have
$$M(F)=R(F)\tensor_{R(L_t)}M(L_t)=\cEi_{L_0/L_t} M(L_t). $$
In view of the resulting formula we call such last-determined modules {\em
 quasicoherent (qc)}, explaining the $qc$ in the definition.  

It is also {\em first-determined} in the sense that for any flag $F$
as above,  
$$M(F)=R(F)\tensor_{R(L_0)} M(L_0). $$
We call such first-determined modules {\em extended (e)}, explaining the $e$ in the definition.  

Because of middle independence, we only need names for first-determined and
last-determined modules. 
\end{remark}

\begin{remark}
As in Remark \ref{rem:phinotflag},  if $H\supset K$  we have maps
$$M(H)\lra M(H\supset K)\lla M(K)$$
but  there is no direct map $M(H)\lra M(K)$. 
\end{remark}

\begin{remark}
(i) By quasicoherence, the values on single element flags
determine all values and therefore  play a special role. Accordingly, we sometimes use the notation 
$$\phi^KM=M(K)$$
to emphasize this.

(ii) The coefficient system $R$ is a module over itself, and as such
we acquire a third notation: $\phi^KR=R(K)=R(G/K)$. The notations
$\phi^KR$ and $R(G/K)$ both have connotations in the equivariant
setting, and the notation here is consistent with \cite{tnq1, tnq2,
  tnqcore}. Indeed, considering the flag $(G\supset 1)$, we obtain
$$\diagram 
M(1)\rto \dto^= & M(G\supset 1) \dto^{\cong}&M(G)\lto \dto^{\cong}\\
M(1)\rto  & \cEi_G M(1) &\phi^G M \lto.
\enddiagram $$
Thus if $G$ is a torus we may consider the poset $\Sigma_a$ as in
Example \ref{eg:BorelTorus} above. Now if  $X$ is a finite $G$-space and the module $M$ is given by Borel
cohomology of fixed points
$$\phi^KM=H^*_{G/K}(X^K), $$
then $M$ is qce by the Borel-Hsiang-Quillen localization theorem; for example,  
$$\cEi_G H^*_G(X) \cong \cEi_G H^*(BG)\tensor_{\Q} H^*(X^G). $$
The definition was originally designed precisely to capture the Localization
Theorem. 
\end{remark}

\section{The category of pairs}
\label{sec:pf}
Suppose we have a splitting system $R^s$ with Euler classes and consider the associated
coefficient system $R^f: \flag (\Sigma)\lra \rings$. 
In view of the fact that the value $R(F)$ depends only on the first
and last term of the flag $F$, most of the coefficient system is rather
redundant, at least when we consider $qc$-modules or $e$-modules. 

Accordingly we may introduce a more economical category to capture this. 

\begin{defn} 
(i) The category of {\em  pairs} $ \qp
(\Sigma)$, which is a partially ordered set with objects the pairs 
$(K\supseteq L)$. The order is given by  $(K\supseteq L) \leq (H\supseteq M)$ if 
$H\supseteq K\supseteq L\supseteq M$, so that there is a
morphism if we increase the first term and decrease the
last. We will use the letter $p$ to indicate the use of pairs.

(ii) The morphisms are composites of the 
 {\em horizontal} morphisms $h: (K\supseteq L)\lra (H\supseteq L)$
 increasing the first term 
and {\em vertical} morphisms $v: (H \supseteq K ) \lra (H \supseteq
L)$ decreasing the last, where  $H \supseteq K\supseteq L$. 
\end{defn}

\begin{remark}
In the terminology of \cite{tnq2}, $\qp (\Sigma)$ would be called the
category of `quotient pairs' and  $(K\supseteq L)$
would be written $(G/K)_{G/L}$; it embodies the $G/L$-equivariant
information in the $L$-fixed points, namely the part that the
localization theorem says should give the $(G/L)/(K/L)$-equivariant
$K/L$-fixed point information. In any case, the value at $(K\supseteq
L)$ only considers information above $L$, and concentrates on  the part coming from above
$K$. 
\end{remark}

Note that $\qp (\Sigma)$ is not simply related to $\flag
(\Sigma)$ since there are no morphisms between two 2-flags. Nonetheless,
because $R$ is a splitting system with Euler classes, it  does define
a $\qp (\Sigma )$-diagram $R^p$ of rings.  Indeed, 
when $H\supseteq K \supseteq L \supseteq M$ we have a commutative square
$$\diagram
\cEi_{K/L} R(G/L)= 
R(K\supseteq L) \rto \dto &R(H\supseteq L) =\cEi_{H/L} R(G/L) \dto \\
\cEi_{K/L} R(G/K)=  R(K\supseteq M) \rto  &R(H\supseteq M) =\cEi_{H/M} R(G/M) 
\enddiagram$$

\begin{defn}
(i) The category of $R^p$-modules is the category of modules over the
$\qp (\Sigma)$-diagram $R^p$ of rings.

(ii) A module $M$ is {\em quasi-coherent (qc)} if the horizontal maps
are given by extensions of scalars, so that if $H\supseteq K \supseteq
L$ then the horizontal structure map induces an isomorphism 
$$R(H\supseteq L)\tensor_{R(K\supseteq L)} M(K\supseteq L) =
\cEi_{H/L}M(K\supseteq L)\stackrel{\cong}\lra 
M(H\supseteq L) .$$

(ii) A module $M$ is {\em extended (e) } if the vertical maps
are given by extensions of scalars, so that if $H\supseteq K \supseteq
L$ then the vertical structure map induces an isomorphism 
$$R(H\supseteq L)\tensor_{R(H\supseteq K)} M(H\supseteq K)
\stackrel{\cong}\lra   M(H\supseteq L) .$$
\end{defn}

We may define a functor 
$$f: \mbox{$R^p$-modules}\lra  \mbox{$R^f$-modules}$$
by selecting just the first and last term of the flag:
$$(fN) (F):=N(f(F)\supseteq l(F)).  $$
For the structure maps we assume $E$ is a subflag of $F=(L_0\supset
\cdots \supset L_t)$ obtained by omitting the single term  $L_i$. 
If $0<i <t$ the structure map is the identity. If
$i=t$ we use the vertical morphism and if $i=0$ we use the horizontal
morphism. 

\begin{lemma}
\label{lem:QPisflag}
The functor $f$ identifies $R^p$-modules as the $R^f$-modules  which are middle-independent in the sense that
inclusions of flags $E \lra F$ induce  isomorphisms if $E $ and $F$
have the same first and last entry.

In particular, it induces equivalences on the subcategories of $e$,
$qc$ and $qce$ modules: 
$$\mbox{$e$-$R^p$-modules}\simeq  \mbox{$e$-$R^f$-modules}$$
$$\mbox{$qc$-$R^p$-modules}\simeq  \mbox{$qc$-$R^f$-modules}$$
$$\mbox{$qce$-$R^p$-modules}\simeq  \mbox{$qce$-$R^f$-modules}$$
\end{lemma}

\begin{proof}
One may define a functor in the opposite direction
$$p: \mbox{middle-independent-$R^f$-modules}\lra 
\mbox{$R^p$-modules}$$ 
on the category of middle-independent
modules. On objects, we simply take $(pM)(K\supseteq L) :=M(K\supseteq
L)$. For $H\supseteq K \supseteq L$ the horizontal and vertical
morphisms are obtained from 
$$M(K\supseteq L)\lra M(H\supseteq K \supseteq L)\stackrel{\cong}\lla M(H\supseteq
L)$$
and 
$$M(H\supseteq K)\lra M(H\supseteq K \supseteq L)\stackrel{\cong}\lla M(H\supseteq L)$$
by inverting the second map. To see this respects compositions, we
compare to higher flags involving all objects involved in the
composition. It is clear that $f$ and $p$ are quasi-inverse. 

Quasi-coherent $R^f$-modules are last-determined in the sense
of the formula $M(F)=R(F)\tensor_{R(l(F))}M(l(F))$; since $R(F)$
is middle-independent,  the quasi-coherent modules are middle-independent. 
Similarly, extended $R^f$-modules are first-determined
and a dual argument applies. 
\end{proof}


\section{Multiplicities}
On some occasions we want to artificially increase the size of our
poset $\Sigma$, constructing a new poset $\Sigmat$ in rather a trivial
way. We will use this to bring the rings occurring in our coefficient
systems under control. 


\begin{defn}
A {\em system of multiplicities}  is a covariant
functor $\cF/: \Sigma \lra \sets$ so that if $i: L\subset K$ then
$i_*: \cF/L \lra \cF/K$ is surjective. We also require that $\cF /G$
is a singleton (also denoted $G$). 
\end{defn}

\begin{example}
(i) If $\Sigma =\connsub (G)$ there is a system of multiplicities given
by  specifying the set of subgroups
$$\cF /K := \{ \tK \st \mbox{ the identity component of $\tK$ is $K$
}\}.$$
If $i: L\subset K $ then the map $i_*:\cF /L \lra \cF /K$ is given by
$i_*(\tL):=\tL \cdot K$. Note that $i_*(\tL)=\tL
\cdot K$ has identity component $K$ and it has $\tL$ as a cotoral subgroup;
it is the unique subgroup with these two properties. To see the map is
surjective, note that  any subgroup $\tK$ is an
internal direct product of $K$ and a finite group $F$, and so $\tK =i_*(L.F)$. 

(ii)  A surjective map of posets $q:\Sigmat\lra \Sigma$ has fibres
$\cF /K=q^{-1}(K)$, but these do not generally form a system of
multiplicities. If we require that the elements of $\cF /K$ are
incomparable for each $K$ then the condition is that given $K\supset
L$, and $\tL$ with $q(\tL)=L$, then there is a unique $\tK \supset \tL$ with
$q(\tK)=K$, and we write $i_*(\tL)=\tK$. (This is a very degenerate case of the
requirement  that $q$ is a Grothendieck opfibration with
cleavage). This defines a functor $\cF / : \Sigma \lra \sets$, and we
require in addition that the  morphisms are surjective.  
\end{example}

Given a poset $\Sigma$ and a system of multiplicities $\cF /$ we may
form a new poset $\Sigmat =\Sigma \cF$ with a surjective poset map $q:\Sigma
\cF \lra \Sigma$ preserving the top and maximal elements. Its objects are pairs $(K, \tK)$ where
$K\in \Sigma$ and $\tK \in \cF /K$. The order relation is given by 
$(L, \tL) \subset (K, \tK )$ if (a) $L\subset K$ and (b) $i_*\tL
=\tK$. Where $K$ can be inferred from $\tK$ (as in the subgroup example),  we may abbreviate
$(K,\tK)$ to $\tK$. Note in particular that for a specified  $K$, the elements
of $\cF /K$ are incomparable. 

We note that this gives an alternative approach to a familiar
example. 

\begin{example}
If we take $\Sigma_c =\connsub(G)$ and $\cF$ to be the system of
subgroups with a given identity component, we recover the toral chain poset:
$$\Sigma_a=\TC (G) =\connsub (G) \cF =\Sigma_c\cF.$$
\end{example}

\subsection{Splitting systems with multiplicities}
Given a splitting  system $R$ and a system of multiplicities $\cF/$,
we may introduce mutiplicities into $R$. 

First we note that any map $q: \Sigmat \lra \Sigma$ lets us define a
$\Sigmat$-splitting system $q^*R$ by $(q^*R)(\tK)=R(q(\tK))$. We may
apply this to $\Sigmat=\Sigma \cF$ and 
 the map $q: \Sigma \cF \lra \Sigma$ defined by $q(K, \tK ) =K$ to 
obtain a $\Sigma \cF$ splitting system by taking 
$$R(G/(K,\tK)):=R (G/K), $$
and using the original inflation maps as structure maps.

We may define a new $\Sigma$-splitting system $R\cF $ by taking
products over the fibres of $q$. Explicitly, we take
$$R\cF (G/K) =(R(G/K))^{\cF/K}. $$
If $i:L\subseteq K$ the inflation map 
$$(R\cF)(G/K) =(R(G/K))^{\cF/K} \lra (R(G/L))^{\cF/L}
=(R\cF) (G/L)$$
is defined as a product of the diagonal inflation maps. To explain,
the map is a product indexed by $\cF /K$. The 
 factor corresponding to $\tK \in \cF /K$ is the map 
$$(R (G/K))^{\{\tK \}}\lra (R(G/L))^{i_*^{-1} (\tK )} $$
whose components are all inflation. This is where we use the surjectivity in
the system of multiplicities. 

\begin{remark}
It is natural to use the notation $(R\Sigma )^s
=q_!q^*R^s$, and we will justify this in due course. However, some
care is necessary, since  the two coefficient systems $(q_!R^s)^f$ and $q_!(R^f)$ are
usually different.
\end{remark}

\subsection{Euler classes on $R\cF$}

We note that once we define a set of Euler classes, 
the splitting system $R\cF$ gives rise to a $\flag (\Sigma)$-coefficient
system $(R\cF)^f$. All such  coefficient systems take the value  $\prod_{\tK \in
  \cF  /K}R(G/K)$ at $K$, but the values elsewhere will depend on
Euler classes. 

To define Euler classes it is natural to assume $q$
takes maximal elements to maximal elements, and use a suitable induced
system of maximally generated Euler classes. We illustrate this in the
topological examples of \cite{tnq1, tnq2, tnqcore}.
At present, there are several candidate constructions corresponding to that for the sphere. 
The purpose of the present subsection is to make these explicit, explain their differences and
identify the topologically relevant one.

\begin{example}
\label{eg:RcRcb}
We consider various examples with $\Sigma_a=\TC (G)$ and 
$\Sigma_c=\connsub (G)$. We have maps of posets 
$$\Sigma_c \stackrel{i}\lra \Sigma_a \stackrel{q}\lra \Sigma_c, $$
so that $\Sigma_c$ is a retract of $\Sigma_a$ and
$\Sigma_a=\Sigma_c\cF$. 

We start with the ordinary Borel splitting system $\RR$ of Example
\ref{eg:BorelTorus}, now introducing decorations so we can introduce
the diagrams from Section \ref{sec:motivation}. To start with, we have
the basic splitting system  
$$\RRa^s (G/K)=H^*(BG/K) $$
on $\Sigma_a$. From this we form $\RRc^s=q_!\RRa^s$ so that 
$$\RRc (G/K) =\prod_{\tK \in \cF /K}H^*(BG/\tK) =\cO_{\cF /K}$$
(where the notation $\cO_{\cF /K} $ is that used in \cite{tnq1, tnq2, tnqcore}). 

We note that we could also  form $i^*\RRa^s$ on $\Sigma_c$ and
introduce multiplicities to form $\RRcb^s=(i^*\RRa^s)\cF$,
where we have
$$\RRcb^s (G/K) =(i^*\RRa^s) \cF (G/K)=\prod_{\tK \in \cF /K}H^*(BG/K).$$

We note that since we are working over the rationals, termwise inflation gives an isomorphism 
$$\RRc^s \stackrel{\cong}\lra \RRcb^s$$
of splitting systems. 
\end{example}

We now consider several choices of maximally generated Euler classes. 

\begin{example}
\label{eg:RRc}
On a codimension 1 connected subgroup $H$, the canonical example $\RRc$
has the value 
$$\RRc (G/H)=\prod_{\tH \in \cF /H} H^*(BG/\tH)=\cO_{\cF /H} .$$
It has Euler classes $c(\alpha)(\tH)=c_1(\alpha^{\tH})$, where 
$\alpha$ runs through all non-trivial one dimensional representations of $G$.
In particular, if $\alpha$ is a faithful representation of $G/H$ then 
$$c(\alpha^n)=\dichotomy{nc_1(\alpha) \mbox{ if } |\tH /H| \;|\: n}
{1\hspace*{4ex} \mbox{ if } |\tH /H| \;\not |\: n}. $$
This is different to the diagram used in the construction onf $\cA_c^f (G)$.
\end{example}

\begin{example}
\label{eg:rrcfisrrcbf}
Now consider $\RRcb$;  on a codimension 1 connected
subgroup $H$ it has the value 
$$\RRcb (G/H)=\prod_{\tH \in \cF /H} H^*(BG/H)$$
and we need to consider how to define Euler classes $c(\alpha ) \in \RRcb (G/H)$.

(i) Diagonal maps give  a map of $\Sigma_c$-splitting systems $i^*\RRa \lra
(i^*\RRa)\cF$. We may use the system of Euler classes from $\RRa $ to give a
system on $(i^*\RRa) \cF$. 

This would mean that we use generating Euler classes defined by 
$c(\alpha) (\tH)=c_1(\alpha^H)$, independent of $\tH$. 
In particular, if $\alpha$ is a faithful representation of $G/H$ then 
$c(\alpha^n)=nc(\alpha)$. In this case we would only need to  use Euler
classes of characters with connected kernel. 

(ii) If we forget the diagonal is available, for each 
$\tH \in \cF /H$ we have an Euler class $c(\alpha)_{\tH}$ whose value
at $\tH'$ is 1 if $\tH'\neq \tH$ and is $c_1(\alpha^H)$ if
$\tH'=\tH$. 

Under the isomorphism $\RRc \cong \RRcb$ described in Example \ref{eg:RcRcb}, 
 this second collection of Euler classes gives the same localization as the
 natural Euler classes of $\RRc$ so we have an isomorphism of
 coefficient systems
$$\RRc^f \cong \RRcb^f. $$ 
Thus we have two slightly different approaches to the diagram of rings
used in constructing $\cA^f_c(G)$.
\end{example}


\section{Change of poset}
\label{sec:changeposetI}
In organizing the information in categories of modules, there is a
balance between the information put into the poset $\Sigma$ and the
information put into the rings. We aim to show that on the categories of
modules of interest to us, we can move between these easily. However
there are a number of different functors that will all be
important. In this section we give an overview. It is easy to break
apart modules over product rings with idempotents, giving a functorial construction
$e$. Depending on the domain and codomain categories, the functor $e$ has
a number of left and right adjoints. In this section we construct the
most obvious adjoint to $e$. In Sections
\ref{sec:pistar} to \ref{sec:Euleradaptedpairs} we construct  other
functors and establish adjunctions that let us work with them.

\subsection{Change of poset for coefficient systems}
\label{sec:pshriekone}
We  start with a surjective poset map $\pi: \Sigma \lra \Sigmab$ which
takes the top element $G$ of $\Sigma$ to the top element $\Gb$ of
$\Sigmab$, and  also takes the set of  maximal elements of $\Sigma$ onto the set of
maximal elements of $\Sigmab$. 

\begin{example}
(i) One example of importance is when we have a dimension function $d: \Sigma \lra I $
where $I=[0,r]=\{ 0, 1, \ldots , r\}$.

(ii) A second example is the map $q: \Sigmab \overline{\cF}\lra
\Sigmab$ arising from a system of multiplicites $\overline{\cF}$ on
$\Sigmab$. This example has special features.
\end{example}

We first note that given a splitting system $\Rb$ on $\Sigmab$ we may
define a splitting system $\pi^*\Rb$ on $\Sigma $ by 
$$(\pi^*\Rb)(K)=\Rb(\pi K).$$
One naturally expects a  right adjoint to this construction to be
given on objects by the formula
$$(\pi_!R)(\Kb)=\prod_{\pi K=\Kb}R(K), $$
but $\pi$ needs to satisfy additional properties before we may define structure
maps. Fortunately, $\pi$ induces a map on flags, and it is
straightforward to observe this has the property we require. 

\begin{lemma}
\label{lem:cleavage}
The map $\pi:\flag (\Sigma)\lra \flag (\Sigmab)$  is a Grothendieck
fibration with cleavage in the sense that given an inclusion  $\Eb\lra \Fb$ of
$\Sigmab$-flags and $F$ with $\pi F=\Fb$, there is a unique $\Sigma$-subflag
$E$ of $F$ with $\pi E=\Eb$. \qqed
\end{lemma}

In this
section we deal with the general framework, and in Section
\ref{sec:Euleradapted} we look at the Euler adapted context which is 
more directly relevant. 

\begin{defn}
Given a surjective map $\pi: \Sigma\lra \Sigmab$ and a coefficient system
$R$ on $\flag (\Sigma)$ we may define a coefficient system $\pi_!R$ on
$\flag (\Sigmab)$ on flags by 
$$(\pi_!R)(\Fb)=\prod_{\pi F=\Fb}R(F). $$
Given a map $\Eb\lra \Fb$ of flags, the map 
$$(\pi_!R)(\Eb)=\prod_{\pi E=\Eb}R(E)\lra \prod_{\pi F=\Fb}R(F)=(\pi_!R)(\Fb)$$
is a product indexed by $E$ with $\pi E=\Eb$ of the maps
$$R(E) \lra \prod_{F\supset E, \pi F=\Fb}R(F)$$
with components coming from the structure maps of $R$. 
\end{defn}

\subsection{Flag idempotents}
\label{subsec:flagidempotents}
First we describe how we may obtain $R$-modules from $\pi_!R$-modules. 

Recall that $lF$ means the last (or smallest) term in the flag $F$.
The key is to note that there is a canonical choice of idempotent
$$e_F\in  \prod_{\pi F=\Fb}R(lF), $$
and  if $E$ is a subflag of $F$ then $e_F$ is a refinement of the image of
$e_E$ in $\prod_{\pi F=\Fb}R(lF)$. This gives compatible idempotents for
all systems of Euler classes. 

\begin{lemma}
If $E$ is a subflag of $F$ with $\pi E=\Eb, \pi F=\Fb$ then 
$e_F(\pi_!R)(\Fb)=R(F)$ and the map 
$$R(E)=e_E(\pi_!R)(\Eb)\lra e_F(\pi_!R)(\Fb)=R(F)$$
coincides with the original structure map of $R$.\qqed  
\end{lemma}

\begin{lemma}
\label{lem:defe}
Applying idempotents gives  a functor 
$$e: \mbox{$\pi_!R^f$-modules}\lra \mbox{$R^f$-modules}.$$
defined by 
$$(e\Mb )(F) =e_F\left[ \Mb (\pi F) \right],   $$
where $\Mb$ is a $\pi_!R^f$-module  and  $e_F\in R(\pi F)$ is the idempotent corresponding to $F$.   
\end{lemma}

\begin{proof}
First, we need to describe the structure maps associated to an
inclusion $E\lra F$ of flags. We have an inclusion $\pi E\lra \pi F$ giving 
$\Mb (\pi E)\lra \Mb (\pi F)$. Since the idempotent the image of $e_E$ in $R(\pi F)$
refines $e_F$ we have an induced map
$$(e\Mb )(E)=e_E\Mb (\pi E)\lra e_E\Mb (\pi F)\lra e_F\Mb (\pi
F)=(e\Mb )(F). $$
These are compatible with the module structure. 
\end{proof}

\subsection{The various adjoints}
This subsection is designed as a guide to the following sections where
a number of different adjoints to $e$ are described. The point is that
the functor $e$ can be viewed as a functor between several different
pairs of categories, and in each case it may have left or right
adjoints. 

We start by assuming that the $\flag (\Sigma)$-diagram of rings $R$ is
given, and we have formed $\flag (\Sigmab)$-diagram $\pi_!R$. 

\begin{itemize}
\item The functor $e:\mbox{$\pi_!R$-modules}\lra \mbox{$R$-modules}$
  has a right adjoint $\pi_!$ consistent with the notation $\pi_!R$
  for coefficient systems (see Subsection \ref{subsec:pishriek}).
\item Given a $\flag(\Sigmab)$ diagram $\pi_!'R $ with a map $\pi_!'R
 \lra \pi_!R$ inducing an isomorphism $e\pi_!'R\cong e\pi_!R$, the functor
$e:\mbox{$\pi_!'R$-modules}\lra \mbox{$R$-modules}$ has a left adjoint
$\pi_*$ (see Section \ref{sec:pistar}). 
\item If $R$ has a system of maximally generated Euler classes, there
 is a $\Sigmab$-diagram $\pi_!^eR$ of rings with a map $\pi_!^e R\lra
 \pi_!R$ which induces an isomorphism   $e\pi_!^eR\cong e\pi_!R$. The functor 
$e:\mbox{$iqc$-$\pi_!^eR$-modules}\lra \mbox{$\pi$-cts-$qc$-$R$-modules}$ has a right adjoint
$\pi_!^e$, where $iqc$ modules are those $M$ for which $eM$ is $qc$, and
where $\pi$ continuity is a notion to be defined below (see Section
\ref{sec:Euleradapted}). 
\item A version of the previous right adjoint with flags replaced by
  pairs (see Section
\ref{sec:Euleradaptedpairs}). 
\end{itemize}

We attempt to use notation that suggests the category of origin. For
example,  $M$ is a module based on a $\Sigma$-diagram of rings, $\Mb$ is 
a module based on a $\Sigmab$-diagram of rings.

\subsection{Modules over $R$ and $\pi_!R$}
\label{subsec:pishriek}
To obtain $\pi_!R$ modules from $R$-modules, we extend the functor $\pi_!$ to modules.

\begin{defn}
\label{defn:pishriek}

(i) For a module $M$ over $R$ we take 
$$(\pi_!M)(\Fb)=\prod_{\pi F=\Fb} M(F)$$
with structure maps given by Lemma \ref{lem:cleavage} as for $\pi_!R$. 

(ii) A $\flag (\Sigmab)$-$\pi_!R$-module $\Mb$ is said to be a {\em
 $p$-module} (or {\em product module}) if  the natural map
$$\Mb(\Fb)\lra \prod_{\pi F=\Fb}e_F\Mb (\Fb)$$
is an isomorphism for all flags $\Fb$. 
\end{defn}

These constructions give the relationship we need between $\flag
(\Sigma)$-$R$-modules and $\flag (\Sigmab)$-$\pi_!R$-modules. 

\begin{lemma}
The constructions $e$ and $\pi_!$ above give an adjunction
$$\adjunction{e}
{\mbox{$\pi_!R^f$-modules}}
{\mbox{$R^f$-modules}}
{\pi_!}. $$

We find $e\pi_!=1$ and the adjunction gives an equivalence
$$\mbox{$p$-$\flag (\Sigmab)$-$\pi_!R$-modules}\simeq \mbox{$\flag (\Sigma)$-$R$-modules}.\qqed $$
\end{lemma}

\begin{remark}
If $\Eb$ is a subflag of $\Fb$ then the structure map $(\pi_!R)(\Eb)\lra
(\pi_!R)(\Fb)$ induces a map 
$$(\pi_!R)(\Fb) \tensor_{(\pi_!R)(\Eb)} (\pi_!M)(\Eb) \lra (\pi_!M)(\Fb). $$
This is a product over flags $E$ with $\pi E=\Eb$ of terms 
$$\left(\prod_{F\geq E, \pi F=\Fb}R(F) \right)\tensor_{R(E)}M(E)\lra
M(F). $$
Note in particular that even if $R(F)\tensor_{R(E)}M(E)\cong M(F)$,
the corresponding statement will usually not hold for the $\pi_!R$-module $\pi_!M$.
\end{remark}

\section{A left adjoint to $e$}
\label{sec:pistar}
In this subsection we again consider a surjective map $\pi :\Sigma
\lra \Sigmab$.  Given a $\Sigma$-diagram of rings $R$, we form the $\flag
(\Sigma)$-diagram $R^f$. We suppose given  a  $\flag(\Sigmab)$-diagram
$\RRbf$ with a map $\RRbf \lra \pi_!R$ which becomes an isomorphism
with  $e$ applied, so  that $e\RRbf=\RRf$. 

Using idempotents as in Subsection 
\ref{subsec:flagidempotents}, and using the fact that $e\RRbf=\RRf$,
we have a functor 
$$e: \mbox{$\RRbf$-modules} \lra \mbox{$\RRf$-modules}. $$
We have already constructed a right adjoint $\pi_!$ to $e$, and in this section we construct  a left adjoint
$$\pi_*: \mbox{$\RRf$-modules}\lra \mbox{$\RRbf$-modules}. $$
We do not display the dependence of this functor on $\RRbf$ in the notation. 

\subsection{Definition of $\pi_*$}
Notationally, we consider  flags $E=(K_0\supset K_1
\supset \cdots \supset K_s)$ and $F=(L_0\supset L_1
\supset \cdots \supset L_t)$ in $\Sigma$, and flags in
$\Sigmab$ will use corresponding barred notation
 so that $\Eb=(\Kb_0\supset \Kb_1\supset \cdots \supset \Kb_s)$ and 
 $\Fb=(\Lb_0\supset \Lb_1\supset \cdots \supset \Lb_t)$. 

We first recall that the {\em right} adjoint $\pi_!$ was defined as
follows: for an $\RRf$-module $X$, the module $\pi_!X$ is
defined  on the flag $\Fb$ using products
$$(\pi_!X)(\Fb)=\prod_{\pi F=\Fb}X(F). $$
We needed the fact that the map of flags was a Grothendieck fibration as stated in Lemma \ref{lem:cleavage} to define the structure maps. The first guess about how
to construct a {\em left} adjoint would be  to replace the product with a sum. This works
if $\Sigma$ is finite, but in general the structure map including a
length 0 flag in a length 1 flag cannot be defined because  
the map  $X(L)\lra  \prod_{K\supset L}X(K\supset L)$ usually fails to  factor through
the sum.

\begin{defn}
\label{defn:pistar}
For each $\RRf$-module $X$, we define  $\pi_*X$ in steps. 
First, on the  flag $\Fb=(\Kb_0\supset \Kb_1\supset \cdots \supset \Kb_s)$ we define $(\pi'_*X)(\Fb)$ to be the sum:
$$(\pi'_*X)(\Fb)=\bigoplus_{\pi F=\Fb}X(F) \subseteq \prod_{\pi F=\Fb}X(F)=(\pi_!X)(\Fb) .$$
The value we want $(\pi_*X)(\Fb)$ lies between the sum and the product
$$(\pi'_*X)(\Fb)=\bigoplus_{\pi F=\Fb}X(F) \subseteq (\pi_*X)(\Fb)\subseteq
\prod_{\pi F=\Fb}X(F) =(\pi_!X)(\Fb) .$$
We take $(\pi_*X)(\Fb)$ to be the $\RRbf (\Fb)$-submodule spanned by
the sum $(\pi_*'X)(\Fb)$ together with the
images of the singleton flags:
$$(\pi_*X)(\Fb)=(\pi'_*X)(\Fb)+\sum_{j=0}^t (\pi_*X)(\Fb,\Lb_j)$$
where 
$$(\pi_*X)(\Fb,\Lb_j)=\RRbf(\Fb)\cdot \pi_!(\Lb_j\lra \Fb)(X(\Lb_j)).$$
\end{defn}

\begin{remark}
It is important that we have
not taken the image of $\pi_!$ but rather the $\RRbf(\Fb)$-submodule
it generates. 
\end{remark}

\begin{lemma}
The structure maps of $\pi_!X$ respect the submodules $(\pi_*X)(\Fb)$, and
hence $\pi_*X$ is an $\RRbf$-module functorially associated to $X$. 
\end{lemma}

\begin{proof}
The additional generators in $\pi_*X$ beyond $\pi'_*X$ all come from
singleton flags, so that the image of any subflag $\Eb$ of $\Fb$ is
contained in the sum of the  images of its terms. 
\end{proof}

\begin{prop}
The functor $\pi_*$ is left adjoint to $e$:
$$\adjunction{\pi_*}
{\mbox{$\RRf$-modules}}{\mbox{$\RRbf$-modules}}
{e}$$
\end{prop}

\begin{proof}
To define the unit $X \lra e\pi_*X$ we need only note that since each $(\pi_*X)(\Fb)$ is
between the sum and the product, we have equality $e\pi_*X=X$. 

The counit $\pi_* e\Xb \lra \Xb$ is taken to be the inclusion, since by
definition, for each flag $\Fb$
$(\pi_*e\Xb)(\Fb)$ is a submodule of $\Xb (\Fb)$. 

The triangular identities are readily verified. 
\end{proof}

\subsection{The functor $\pi_*$ on qce-modules}
In this section we suppose given a $qce$-$\RRf$-module $M$, and we
consider the behaviour of $\pi_*$ on $X=\iota M$, where $\iota$ is the functor
including $qce$-modules in all modules. 

\begin{lemma}
If $X=\iota M$ for a $qce$-module $M$, then 
the submodule $(\pi_*X)(\Fb,\Lb_j)$ contains (as a retract) each of the submodules $X(F)$
with $\pi F=\Fb$.
\end{lemma}

\begin{proof}
Suppose $F=(L_0\supset L_1\supset \cdots \supset L_t)$.
Note that we have an idempotent $e_F$ in $\RRbf (\Fb)$ so that for any
$\RRbf$-module $X$,  the image
of $X(\Lb_j)$ in $X(\Fb)$ contains the image of $e_{L_j}X(\Lb_j)$ in
$e_FX(\Fb)$. 

Now observe  that if $M$ is $qce$ then the image of any $M(L_j)$ in $M(F)$
generates $M(F)$ as an $\RRf (F)$-module. 
\end{proof}

The submodule $(\pi_*X)(\Fb) \subseteq (\pi_!X)(\Fb)$ is obtained by
permitting elements with infinitely many non-zero terms when they
occur along certain specific diagonals. However, as we saw in the previous lemma, the
diagonal elements automatically lead to the inclusion of elements with only
finitely many terms. To get the combinatorics under control, 
we  consider intersections of the submodules $(\pi_*X)(\Fb, \Lb_j)$. 
For the subflag, $\Eb =(\Kb_0\supset \cdots \supset \Kb_s)\subseteq \Fb$ we take
$$(\pi_*X)(\Fb,\Eb)=\bigcap_{i}(\pi_*X)(\Fb,\Kb_i) .$$

\begin{remark}
In applications, we need to consider differential graded objects $X$, and the 
Mayer-Vietoris spectral sequence gives a means of calculating the
homology of a complex $(\pi_*X)(\Fb)$ from those of the intersections. 
\end{remark}

\begin{lemma}
If $X=\iota M$ for a $qce$-module $M$ then for flags $\Eb \subseteq \Fb$ of
$\Sigmab$, 
$$(\pi_*X)(\Fb,\Eb)=\bigoplus_{\pi E=\Eb} (\pi_*X)(\Fb,\Delta E) , $$
where 
$$(\pi_*X)(\Fb,\Delta E) = \RRbf (\Fb) \cdot \im \left[ X(E)\stackrel{\Delta}\lra (\pi_!X)
  (\Fb)\right] .$$ 
\end{lemma}

\begin{proof}
In view of the intersection result, it suffices to prove the result
for the singleton subflags $\Eb=\Lb_j$. 

Note that since $(\pi_*X)(\Fb, \Lb_j)$ is the image of a map from a sum
of the terms $X(L_j)$ with $\pi L_j=\Lb_j$ the image is a corresponding
sum. This gives the first equality
$$(\pi_*X)(\Fb,\Lb_j) =\sum_{\pi L_j=\Lb_j}(\pi_*X)(\Fb,L_j)
=\bigoplus_{\pi 
  L_j=\Lb_j}(\pi_*X)(\Fb,L_j); $$
the sum is direct,  since the term $(\pi_*X)(\Fb,L_j)$ is only
non-zero in the $F$-components if the  flag $F$ contains $L_j$. 
\end{proof}

\section{Euler-adapted change of poset for coefficient systems}
\label{sec:Euleradapted}

We continue with the notation of Section \ref{sec:changeposetI} with a
splitting system $R^s$ on $\Sigma $ giving a coefficient system $R^f$ on
$\flag (\Sigma) $ and  a map $p:\Sigma \lra \Sigmab$.  We now suppose that $R$ is equipped with  maximally generated Euler
classes, and that $\pi:\Sigma \lra \Sigmab$ takes  top and maximal elements to top and maximal elements. 

In Subsection \ref{sec:pshriekone} we constructed a right adjoint
functor $\pi_!$ to $e$ on coefficient systems and
on modules, and in this subsection we describe a variant $\pes$
suitable for quasi-coherent modules in  which the Euler classes are taken from $\Sigma$. 

\subsection{The Euler-adapted construction}

To start with, the coefficient system agrees with $\pi_!R $ on vertices 
$$(\pes R)(\Kb) =\prod_{\pi K=\Kb} R(K).  $$

\begin{defn}
(i) In a maximally generated system of Euler classes, any $e \in
\cE_{K}$ may be written as a finite product $e=\prod_ie_i$ where
$e_i=\infl_{G/H_i}^{G/K}e'_i$ with $H_i$ maximal and $K\not \subseteq
H_i$. We may then write $e_L$ for the product of those $e_i$ with
$H_i\supset L$. 

If we have some set of subgroups $L\subseteq K$ then we may define
$$\cEi_K \prod_LM(L) =\colim_{e\in \cE_{K}}\prod_L\left[
M(L)\stackrel{e_L}\lra M(L) \right]. $$

(ii) We define a coefficient system $\pes R$ on $\flag (\Sigmab)$ as follows. 
If $\Fb =(\Lb_0\supset \Lb_1 \supset \cdots \supset \Lb_t)$ we take
$$(\pes R)(\Fb)=\prod_{p (L_0)=\Lb_0} \cEi_{L_0} \prod_{\pi (L_1)=\Lb_1,
L_1 \subseteq L_0} \cdots \cEi_{L_{t-1}}\prod_{\pi(L_t)=\Lb_t,
L_t\subseteq L_{t-1}}  R(G/L_t).$$

If we have an inclusion $a: \Eb\lra  \Fb$ of flags we need to describe the induced map 
$$(\pes R)(a): (\pes R)(\Eb)\lra (\pes R)(\Fb).$$
 It suffices to do this when $\Eb$ is obtained
by omitting one factor, so we suppose $\Fb=(\Lb_0\supset \cdots \supset
\Lb_t)$
and that $\Eb$ omits $\Lb_j$. 

If $j=t$ then we first describe $(\pes R)(\Lb_{t-1})\lra (\pes R)(\Lb_{t-1}\supset
\Lb_t)$. We take the product of factors indexed by $L_{t-1}$ with
$\pi L_{t-1}=\Lb_{t-1}$; the $L_{t-1}$ factor is the map
$$R(G/L_{t-1}) \lra \prod_{\pi L_t=\Lb_t, L_t\subset
  L_{t-1}}R(G/L_t)\lra \cEi_{L_{t-1}}\prod_{\pi L_t=\Lb_t, L_t\subset L_{t-1}}R(G/L_t)$$
where the first map has  components which are the inflations from
$G/L_{t-1} $ to $G/L_t$ and the second map inverts $\cE_{L_{t-1}}$. 
To obtain the map for $a: \Eb\lra \Fb$ we apply the sequence of
localizations and products to each term.

If $j<t$   then we apply an operation to $R^{\dagger}=(\pes R)(\Lb_{j+1}\supset \cdots
\supset \Lb_t)$. Indeed, adding $\Lb_j$ to the flag on the codomain we
need a map
$$R^{\dagger}\lra \prod_{\pi L_j=\Lb_j}\cEi_{L_j}R^{\dagger}=(\pes R)(\Lb_j\supset \Lb_{j+1}\supset
\cdots \supset \Lb_t),$$
and we use the map whose components are the localizations. 
To obtain the map for $a:\Eb\lra \Fb$ we apply the sequence of
localizations and products to each term.
\end{defn}

\begin{remark}
(i) If the localizations all involved inverting only units, we could omit
$\cEi$ everywhere and find
$R(F)=R(G/L_t)$, and  $(\pes R)(\Fb)=\prod_{\pi F=\Fb}R(F)$. When we invert
non-units, the localizations for
the $\flag (\Sigma)$ system just accumulate, but those for the $\flag
(\Sigmab)$ impose a continuity condition related to the finiteness of
the fibres of $\pi$. The statement of Lemma \ref{lem:pshrieks} below is a
stronger variant of this.

(ii) We are assuming that the number of maximal elements of $\Sigma$
and the number of maximal generators are countable. To calculate the direct limit in the first part of the definition
we may choose an ordering on the maximal elements $H$ not  containing $K$ and the maximal
generators and then order the elements $e$ accordingly. The colimit is
independent of this ordering. 

(iii) The coefficient system on $\flag (\Sigmab)$ differs from the coefficient system on $\flag (\Sigma)$
in that   the maps $\partial_iF\subseteq F$ will usually not induce the
identity. This is partly because the number flags $E$ over a subflag
$\Eb$ of $\Fb$ will depend on $\Eb$, and partly because of the
relationship between the localization of a product and
the product of localizations. 
\end{remark}

\subsection{Relationship between $\pi_!$ and $\pes$}
First we note that the coefficient rings $\pi_!R$ and $\pes R$ are closely
related. 

\begin{lemma}
\label{lem:pshrieks} 
There is a map $\pes R \lra \pi_!R$ of coefficient systems on $\flag
(\Sigmab)$ which is the identity on flags of length 0. 
\end{lemma}

\begin{proof}
Writing $\cN_L$ as a typographical placeholder for the identity
functor, we see the universal properties of localization give maps 
$$\begin{array}{rcl}
(\pes R)(\Fb)&=&\prod_{\pi (L_0)=\Lb_0} \cEi_{L_0} \prod_{\pi (L_1)=\Lb_1,
L_1 \subseteq L_0} \cdots \cEi_{L_{t-1}}\prod_{\pi(L_t)=\Lb_t,
L_t\subseteq L_{t-1}}  R(G/L_t)\\
\downarrow &&\downarrow\\
(\pi_!R)(\Fb)&=&\prod_{\pi (L_0)=\Lb_0} \cN_{L_0} \prod_{\pi (L_1)=\Lb_1,
L_1 \subseteq L_0} \cdots \cN_{L_{t-1}}\prod_{\pi(L_t)=\Lb_t,
L_t\subseteq L_{t-1}}  \cEi_{L_0/L_t} R(G/L_t)
\end{array}$$
\end{proof}

The fact that the idempotents came from an unlocalized product means
$\pes$ inherits the idempotent properties of $\pi_!$. 

\begin{lemma}
\label{lem:idempotents}
The map of Lemma \ref{lem:pshrieks} is compatible with idempotents;
indeed $e_F(\pi_!R)(\Fb) =R(F)=e_F(\pes R)(\Fb)$ so that applying $e$ to 
$\pes R \lra \pi_! R$ we obtain the identity. \qqed 


\end{lemma}



Using the idempotents introduced in Lemma \ref{lem:idempotents} we may define a functor as follows. 

\begin{lemma} 
Extending scalars to $\pi_!R$ and applying idempotents gives 
a functor 
$$e: \mbox{$\pes R^f$-modules}\lra \mbox{$R^f$-modules}.$$ 
given  by 
$$(e\Mb)(F) =e_F\left[ \Mb (\pi F) \right],   $$ 
where $e_F\in R(\pi F)$ is the idempotent corresponding to $F$. \qqed  
\end{lemma}

\subsection{Relative continuity}
\label{subsec:relctty}

We will define an Euler-compatible right adjoint to $e$. This involves
retricting  the $R$-modules to be compatible with Euler classes in the
sense that they are quasi-coherent (or last-determined), and so that 
they are compatible with $p$.  In effect, the poset structure on
$\Sigmab$ specifies a  topology  on $\Sigmab$ (open sets generated by 
the sets $V(L):= \{ K\st K\supseteq L\}$ of elements above an
element),  and we may imagine that the fibres of $\pi$
specify infinitesimal neighbourhoods of points of $\Sigmab$. This
`topological' structure is then inherited by $\flag (\Sigma)$ and
$\flag (\Sigmab)$. The additional continuity condition explains how
the points of the infinitesimal neighbourhoods approach the limit point. 

\begin{defn}
\label{defn:pishrieke}
If we are given an $R$-module $M$ over $\flag (\Sigma)$ we may consider
its {\em $\pi$-continuous} sections over a flag in $\Sigmab$. Indeed, 
 if $\Fb=(\Lb_0\supset \cdots \supset
\Lb_t)$, we take
$$(\pes M)(\Fb)=M (\Fb)_c=\prod_{\pi (L_0)=\Lb_0} \cEi_{L_0} \prod_{\pi (L_1)=\Lb_1,
L_1 \subseteq L_0} \cdots \cEi_{L_{t-1}}\prod_{\pi(L_t)=\Lb_t,
L_t\subseteq L_{t-1}}  M (L_t).$$
This is evidently a module over 
$$(\pes R)(\Fb)=R(\Fb)_c=\prod_{\pi (L_0)=\Lb_0} \cEi_{L_0} \prod_{\pi (L_1)=\Lb_1,
L_1 \subseteq L_0} \cdots \cEi_{L_{t-1}}\prod_{\pi(L_t)=\Lb_t,
L_t\subseteq L_{t-1}}  R (L_t).$$
\end{defn}

As things stand, there is no reason why the structure maps of
$M$ should take continuous sections to continuous sections if the last
term in the flag changes. 

\begin{defn}
\label{defn:pistructure}
A {\em $\pi$-structure} on an $R^f$-module $M$ is a transitive choice of liftings
for each $\Kb\supset \Lb$ and $K$ with $\pi K=\Kb$:
$$\diagram & \cEi_K \prod_{\pi L =\Lb,L\subset K} M(L) \dto  \\
M(K) \rto \ar@{-->}[ur] & 
\prod_{\pi L =\Lb,L\subset K} \cEi_{K/L}M(L) 
\enddiagram$$
Using the language of continuous sections, this can be written in the form
$$\diagram & e_KM(\Kb \supset \Lb)_c\dto\\
M(K) \rto \ar@{-->}[ur] & 
\prod_{\pi L =\Lb,L\subset K} e_{K\supset
 L}M(\Kb\supset \Lb)_c
\enddiagram$$

A map of $R^f$-modules is compatible with $\pi$-structure if it commutes
with the chosen liftings. We write $\pi$-cts-$R$-modules for the category of these. 
\end{defn}

\begin{remark}
Perhaps the best way to formalize this structure and to make the
statement of transitivity clear is to say that a $\pi$-structure is a
map $M\lra \pi_!^eeM$. This turns out to be the unit of an adjunction,
which then explains the role of $\pi$-structures. 
\end{remark}


It seems that a $\pi$-structure is quite subtle in general, but there is
a simple source of $\pi$-structures important in our applications. 

\begin{lemma}
\label{lem:mininducespstructure}
If $\Sigma$ has a bottom element 1 then any quasi-coherent $\flag
(\Sigma)$ module $M$ has a canonical $\pi$-structure. 
\end{lemma}

\begin{proof} In the following diagram, $K$ is fixed as $\Lb\subseteq
  \Kb=\pi K$. The subgroups $L$ in the products run through subgroups $L
\subset K$ with $\pi L=\Lb$. Apart from the two diagonal maps, all maps
come by extension of scalars from the structure maps of $M$. The
isomorphisms come from quasicoherence

$$\diagram
R(K\supset 1)\tensor_{R(K)}M(K)\rto \dto^{\Delta} &M(K\supset
1) \dto^{\Delta} \\
\prod_LR(K\supset 1) \tensor_{R(K)} M(K)\rto  \ar[d]^{\cong} &
\prod_LM(K\supset 1) \dto^{\cong} \\
\prod_L R(L\supset 1)\tensor_{R(L)}R(K\supset L)\tensor_{R(K)} M(K) \ddto
&\prod_LR(K\supset L) \tensor_{R(L)}M(L\supset 1) \\
&\prod_L R(K\supset L) \tensor_{R(L)} R(L\supset 1) \tensor_{R(L)} M(L)\uto
\dto^{\cong}\\
\prod_L R(L\supset 1)\tensor_{R(L)} M(K\supset L)
&
\prod_LR(L\supset 1)\tensor_{R(L)}R(K\supset L)\tensor_{R(L)} M(L)\lto^{\cong}
\enddiagram$$
The required lift arises since the top left vertical takes values in
$R(K\supset 1)\tensor_{R(K)}\prod_L
M(K)$. 
\end{proof}

Reassured by the fact $\pi$-structures can arise naturally, we may proceed. 

\begin{lemma}
There is a functor 
$$\pes : \mbox{$qc$-$\pi$-cts-$R^f$-modules}\lra
\mbox{$\pes R^f$-modules}, $$
defined on vertices by 
$$(\pes M)(\Kb)=(\pi_!M)(\Kb)=\prod_{\pi K=\Kb} M(K). $$
\end{lemma}

\begin{proof}
Since $qc$-modules are last-determined, it is reasonable to 
 extend the definition on vertices to the entire flag complex by
concentrating on the last term in the flag; 
Thus if $\Fb=(\Lb_0\supset \cdots \supset
\Lb_t)$, we take
$$(\pes M)(\Fb)=\prod_{\pi (L_0)=\Lb_0} \cEi_{L_0} \prod_{\pi (L_1)=\Lb_1,
L_1 \subseteq L_0} \cdots \cEi_{L_{t-1}}\prod_{\pi(L_t)=\Lb_t,
L_t\subseteq L_{t-1}}  M (L_t).$$
This is evidently a module over 
$$(\pes R)(\Fb)=\prod_{\pi (L_0)=\Lb_0} \cEi_{L_0} \prod_{\pi (L_1)=\Lb_1,
L_1 \subseteq L_0} \cdots \cEi_{L_{t-1}}\prod_{\pi(L_t)=\Lb_t,
L_t\subseteq L_{t-1}}  R (L_t).$$
If we have an inclusion $i: \Eb\lra  \Fb$ of flags we need to describe the induced map 
$(\pes M)(i): (\pes M)(\Eb)\lra (\pes M)(\Fb)$. It suffices to do this when $\Eb$ is obtained
by omitting one factor, so we suppose $\Fb=(\Lb_0\supset \cdots \supset
\Lb_t)$ and that $\Eb$ omits $\Lb_i$. 

If $i=t$ then we first describe $(\pes M)(\Lb_{t-1})\lra (\pes M)(\Lb_{t-1}\supset
\Lb_t)$. This is 
$$\prod_{\pi L_{t-1}=L_{t-1}'}M(L_{t-1}) \lra \prod_{\pi L_{t-1}=L_{t-1}'}\cEi_{L_{t-1}}
\prod_{\pi L_t=\Lb_t, L_t\subset L_{t-1}}M(L_t), $$
and is a product of factors indexed by $L_{t-1}$ with
$\pi L_{t-1}=\Lb_{t-1}$. The $L_{t-1}$ factor is the map
$$M(L_{t-1}) \lra \cEi_{L_{t-1}}\prod_{\pi L_t=\Lb_t, L_t\subset
 L_{t-1}}M(L_t)$$
given by the $\pi$-structure.

To obtain the map for the full length flags $\Eb\lra \Fb$ we apply the sequence of
localizations and products to the above shortened flags. 

If $i<t$   then we apply an operation to $M^{\dagger}=(\pes M)(\Lb_{i+1}\supset \cdots
\supset \Lb_t)$. Indeed, adding $\Lb_i$ on the codomain we have
$$M^{\dagger}\lra \prod_{\pi L_i=\Lb_i}\cEi_{L_i}M^{\dagger}=(\pes M)(\Lb_i\supset \Lb_{i+1}\supset
\cdots \supset \Lb_t)$$
where we use the map whose components are the localizations. 
To obtain the map for $\Eb\lra \Fb$ we apply the sequence of
localizations and products to each term.
\end{proof}

\begin{remark}
\label{rem:pesqc}
The functor $\pes$ takes values in the $\pes R$-modules $\Mb$ which are
themselves $i$-quasi-coherent in the sense that $e\Mb$  is quasi-coherent
($i$ for idempotent). 

Indeed, 
$$(e\pes M) (F)=e_F(\pes M)(\pi F)=\cEi_{L_0/L_1}\cEi_{L_1/L_2}\ldots
\cEi_{L_{t-1}/L_t}M(L_t)=\cEi_{L_0/L_t}M(L_t)=M(F). $$
We will seek a right adjoint on the restricted category of
$i$-quasi-coherent modules. 
\end{remark}

\subsection{An Euler adapted right adjoint extending modules from $\flag (\Sigma)$ to $\flag (\Sigmab)$}

We are ready to explain the universal property of $\pes$. 

\begin{lemma}
\label{lem:pshriek}
There is an adjunction
$$\adjunction{e}{\mbox{$iqc$-$\pes R^f$-modules}}
{\mbox{$qc$-$\pi$-cts-$R^f$-modules}}{\pes} $$
\end{lemma}

\begin{proof}
The counit $e\pes M\lra M$ is described and seen to be an isomorphism
in Remark \ref{rem:pesqc}. 

On the other hand we obtain a natural map $\Mb\lra \pes e\Mb$ which at $\Fb$ is 
\begin{multline*}
\Mb (\Fb)\lra \prod_{\pi F=\Fb}e_F\Mb(\Fb)  \lra \prod_{\pi F=\Fb}e_FR(F)\tensor_{R(\Lb_t)} \Mb(\Lb_t) \\
\lra \prod_{\pi (L_0)=\Lb_0} \cEi_{L_0} \prod_{\pi (L_1)=\Lb_1,
L_1 \subseteq L_0} \cdots \cEi_{L_{t-1}}\prod_{\pi(L_t)=\Lb_t,
L_t\subseteq L_{t-1}}  e_{L_t}\Mb (\Lb_t)\\
=(\pes e\Mb)(\Fb)
\end{multline*}

These together satisfy the triangular identities and give an
adjunction. 
\end{proof}



\subsection{$p$-modules}

The adjunction in Lemma \ref{lem:pshriek} shows that the category of
$qc$-$\pi$-cts-$R^f$-modules is a retract of the category of 
$qc$-$R^f$-modules.

\begin{defn}
We say that a module over $\pes R$ is a pqc-module if
the natural map $\Mb \lra \pes e \Mb$  made explicit in Lemma
 \ref{lem:pshriek} gives an isomorphism
$$\Mb (\Fb)\cong \prod_{\pi (L_0)=\Lb_0} \cEi_{L_0} \prod_{\pi (L_1)=\Lb_1,
L_1 \subseteq L_0} \cdots \cEi_{L_{t-1}}\prod_{\pi(L_t)=\Lb_t,
L_t\subseteq L_{t-1}}  e_{L_t}\Mb (\Lb_t)$$
\end{defn}

\begin{remark}
(i) For length 0 flags, the condition states that the values on
are simply the products of the values of $e\Mb$ on the length 0 flags of
$\Sigma$:
$$\Mb(\Kb)=\prod_{\pi K=\Kb} e_K\Mb(\Kb). $$
This accounts for the letter $p$ in $pqc$.

(ii) We have already noted that $\pes e\Mb$ is an $iqc$-module. 
This means that  if $\Eb$ and $\Fb$ have the same last terms, and $E\leq F$ with 
$\pi E=\Eb, \pi F=\Fb$ then for any $pqc$-module $\Mb$, writing $f$
for the first (largest) term in a flag, we have 
$$e_F \Mb(\Fb)=\cEi_{fF/fE}e_E\Mb(\Eb). $$
This accounts for the letters $qc$ in $pqc$.
\end{remark}

\begin{defn}
We say that a $\pes R$-module $\Mb$  is {\em $i$-extended} (or an
$ie$-module) if $e\Mb$ is extended as an $R$-module. 

More explicitly, the inclusion of the flag  $\Eb$ into  $\Fb$ induces $\Mb(\Eb)\lra \Mb(\Fb)$ and hence 
$R(\Fb)\tensor_{R(\Eb)} \Mb(\Eb)\lra \Mb(\Fb)$. If $\Mb$ is extended and we take
$\Eb=(\Kb), \Fb=(\Kb\supset \Lb)$ and choose $K\supset L$ with $\pi K=\Kb,
pL=\Lb$ then we obtain an isomorphism 
$$e_{K\supset L} \Mb(\Kb\supset \Lb)= R(K\supset L) \tensor_{R(K)} e_{K}\Mb(\Kb). $$

For brevity we write $pqc$ of product modules which are $i-qc$, and
$pqce$ for product modules that are $i-qce$, since in the presence of
the $p$ condition, the $i$ requirement is the only appropriate choice. 
\end{defn}

\begin{cor}
\label{cor:RmodpesRmod}
The adjunction of Lemma \ref{lem:pshriek}
gives an equivalence
$$\mbox{$pqc$-$\pes R^f$-modules} \simeq
\mbox{$qc$-$\pi$-cts-$R^f$-modules} , $$
and this restricts to an equivalence
$$\mbox{$pqce$-$\pes R^f$-modules} \simeq
\mbox{$qce$-$\pi$-cts-$R^f$-modules} .\qqed $$
\end{cor}

\section{Euler-adapted change of poset for pair systems}
\label{sec:Euleradaptedpairs}

The purpose of this section is to record the Euler-adapted change of
poset for systems of pairs. The proofs are essentially specializations of
those  for flags, so we will not give full details.

\subsection{The Euler-adapted construction}

As before the Euler-adapted construction is a product on vertices. 

\begin{defn}
We define a coefficient system $\pes R$ on $\qp (\Sigmab)$ as
follows:
$$(\pes R)(\Kb\supseteq \Lb)=\prod_{\pi (K)=\Kb} \cEi_{K} \prod_{\pi (L)=\Lb,
L \subseteq K}   R(G/L).$$

We need to describe the induced maps, and it suffices to do this for the horizontal and vertical cases. 
If we have $\Hb\supseteq \Kb \supseteq \Lb$, we have the horizontal
inclusion  $h: (\Kb\supseteq \Lb ) \lra (\Hb\supseteq \Lb)$ and the vertical 
inclusion  $v: (\Hb\supseteq \Kb ) \lra (\Hb\supseteq \Lb)$. 

Starting  with the horizontal map, $(\pes R)(h): (\pes R)(\Kb\supseteq
\Lb)\lra (\pes R)(\Hb\supseteq \Lb)$, we need to define 
$$\prod_{\pi (K)=\Kb} \cEi_{K} \prod_{\pi (L)=\Lb,
L \subseteq K}   R(G/L)\lra  \prod_{\pi (H)=\Hb} \cEi_{H} \prod_{\pi (L)=\Lb,
L \subseteq H}   R(G/L).$$
Using Lemma \ref{lem:cleavage}, for each  particular $K$ with $\pi K =\Kb$, there is a unique $H$
with $\pi H = \Hb$ and $H\supseteq K$. Accordingly we may take a
product indexed by $K$ of maps 
$$\cEi_{K} \prod_{\pi (L)=\Lb,
L \subseteq K}   R(G/L)\lra  \cEi_{H} \prod_{\pi (L)=\Lb,
L \subseteq H}   R(G/L); $$
this includes the smaller product (over $L\subseteq K$) in the larger one
(over $L\subseteq H$) and localizes. 

Moving on to the vertical map, $(\pes R)(v): (\pes R)(H\supseteq
K)\lra (\pes R)(H\supseteq L)$, we need to define 
$$\prod_{\pi (H)=\Hb} \cEi_{H} \prod_{\pi (K)=\Kb,
K\subseteq H}   R(G/K)\lra  \prod_{\pi (H)=\Hb} \cEi_{H} \prod_{\pi (L)=\Lb,
L \subseteq H}   R(G/L).$$
This is a product over $H$ of localizations of 
$$\prod_{\pi (K)=\Kb,
K\subseteq H}   R(G/K)\lra \prod_{\pi (L)=\Lb,
L \subseteq H}   R(G/L).$$
Using Lemma \ref{lem:cleavage}, for each  particular $L$ with $\pi L =\Lb$, there is a unique $K$
with $\pi K = \Kb$ and $L\supseteq K$. Accordingly we may take a
product indexed by $K$ of maps 
$$  R(G/K)\lra  \prod_{\pi (L)=\Lb,
L \subseteq H}   R(G/L),  $$
whose components are inflations. 
\end{defn}

\subsection{Relationship between $\pi_!$ and $\pes$}
The following three lemmas are precisely like the case of flags. 

\begin{lemma}
\label{lem:pshrieksp} 
There is a map $\pes R \lra \pi_!R$ of coefficient systems on $\qp
(\Sigmab)$ which is the identity on flags of length 0. \qqed
\end{lemma}

\begin{lemma}
\label{lem:idempotents}
The map of Lemma \ref{lem:pshrieksp} is compatible with idempotents;
indeed $e_{K\supseteq L}(\pi_!R)(\Kb\supseteq \Lb) =R(K\supseteq
L)=e_{K\supseteq L}(\pes
R)(\Kb\supseteq \Lb)$ so that applying $e$ to 
$\pes R \lra \pi_! R$ we obtain the identity. \qqed 
\end{lemma}

\begin{lemma} 
Extending scalars to $\pi_!R$ and applying idempotents gives 
a functor 
$$e: \mbox{$\pes R^p$-modules}\lra \mbox{$R^p$-modules}.$$ 
given  by 
$$(e\Mb)(K\supseteq L) =e_{K\supseteq L}\left[ \Mb(\pi (K\supseteq L)) \right],   $$ 
where $e_{K\supseteq L}\in R(\pi (K\supseteq L))$ is the idempotent corresponding to
$(K\supseteq L)$. \qqed  
\end{lemma}

\subsection{Relative continuity}
The relative continuity condition for coefficient systems is already
formulated for pairs, so we refer the reader to Subsection
\ref{subsec:relctty}.

\begin{lemma}
There is a functor 
$$\pes : \mbox{$qc$-$\pi$-cts-$R^p$-modules}\lra
\mbox{$\pes R^a$-modules}, $$
defined on vertices by 
$$(\pes M)(\Kb)=(\pi_!M)(\Kb)=\prod_{\pi K=\Kb} M(K). $$
\end{lemma}

\begin{proof}
Since $qc$-modules are last-determined, it is reasonable to 
 extend the definition on vertices to the entire flag complex by
concentrating on the last term in the flag; 
Thus  we take
$$(\pes M)(\Kb\subseteq \Lb)=\prod_{\pi (K)=\Kb} \cEi_{K} \prod_{\pi(L)=\Lb,
L\subseteq K}  M (L).$$
This is evidently a module over
$$(\pes R)(\Kb\subseteq \Lb)=\prod_{\pi (K)=\Kb} \cEi_{K} \prod_{\pi(L)=\Lb,
L\subseteq K}  R (L).$$

We need to describe the induced maps, and it suffices to do this for the horizontal and vertical cases. 
If we have $\Hb\supseteq \Kb \supseteq \Lb$, we have the horizontal
inclusion  $h: (\Kb\supseteq \Lb ) \lra (\Hb\supseteq \Lb)$ and the vertical 
inclusion  $v: (\Hb\supseteq \Kb ) \lra (\Hb\supseteq \Lb)$. 

Starting  with the horizontal map, $(\pes M)(h): (\pes M)(\Kb\supseteq
\Lb)\lra (\pes M)(\Hb\supseteq \Lb)$, we need to define a map 
$$\prod_{\pi (K)=\Kb} \cEi_{K} \prod_{\pi (L)=\Lb,
L \subseteq K}   M(L)\lra  \prod_{\pi (H)=\Hb} \cEi_{H} \prod_{\pi (L)=\Lb,
L \subseteq H}   M(L).$$
Using Lemma \ref{lem:cleavage}, for each  particular $K$ with $\pi K =\Kb$, there is a unique $H$
with $\pi H = \Hb$ and $H\supseteq K$. Accordingly we may take a
product indexed by $K$ of maps 
$$\cEi_{K} \prod_{\pi (L)=\Lb,
L \subseteq K}   M(L)\lra  \cEi_{H} \prod_{\pi (L)=\Lb,
L \subseteq H}   M(L); $$
this includes the smaller product (over $L\subseteq K$) in the larger one
(over $L\subseteq H$) and localizes. 

Moving on to the vertical map, $(\pes M)(v): (\pes M)(\Hb\supseteq
\Kb)\lra (\pes M)(\Hb\supseteq \Lb)$, we need to define 
$$\prod_{\pi (H)=\Hb} \cEi_{H} \prod_{\pi (K)=\Kb,
K\subseteq H}   M(K)\lra  \prod_{\pi (H)=\Hb} \cEi_{H} \prod_{\pi (L)=\Lb,
L \subseteq H}   M(L).$$
This is a product over $H$, and it suffices to construct 
$$\prod_{\pi (K)=\Kb,
K\subseteq H}   M(K)\lra \cEi_H \prod_{\pi (L)=\Lb,
L \subseteq H}   M(L).$$
Using Lemma \ref{lem:cleavage}, for each  particular $L$ with $\pi L =\Lb$, there is a unique $K$
with $\pi K = \Kb$ and $L\supseteq K$. Accordingly we may use the
$\pi$-structure to obtain a map 
$$  M(K)\lra  \cEi_K\prod_{\pi (L)=\Lb,
L \subseteq K}   M(L).  $$
Now take products, localize further to replace $\cEi_K$ by $\cEi_H$ in
every factor, and include the smaller product (over $L\subseteq K$)
into the larger product (over $L\subseteq H$). 
\end{proof}

\begin{remark}
\label{rem:pesqcp}
The functor $\pes$ takes values in the $\pes R$-modules $\Mb$ which are
themselves  i-quasi-coherent in the sense that $e\Mb$  is quasi-coherent. 

Indeed, 
$$(e\pes M) (K\supseteq L)=e_{K\supseteq L}(\pes M)(\pi K\supseteq \pi
L)=\cEi_{K/L}M(L)=M(K\supseteq L). $$
We will seek a right adjoint on the restricted category of
i-quasi-coherent modules. 
\end{remark}

\subsection{An Euler adapted right adjoint extending modules from $\qp (\Sigma)$ to $\qp (\Sigmab)$}

We are ready to explain the universal property of $\pes$. 

\begin{lemma}
\label{lem:pshriekp}
There is an adjunction
$$\adjunction{e}{\mbox{$qc$-$\pes R^p$-modules}}
{\mbox{$qc$-$\pi$-cts-$R^p$-modules}}{\pes} $$
\end{lemma}

\begin{proof}
The counit $e\pes M\lra M$ is described and seen to be an isomorphism
in Remark \ref{rem:pesqcp}. 

On the other hand we obtain a natural map $\Mb\lra \pes e\Mb$ which at
$\Kb\supseteq \Lb$ is 
\begin{multline*}
\Mb (\Kb\supseteq \Lb)\lra \prod_{\pi K=\Kb, \pi L=\Lb}e_{K\supseteq
  L}\Mb(\Kb\supseteq \Lb)  \lra \prod_{\pi K=\Kb, \pi L=\Lb}e_{K\supseteq
  L}R(K\supseteq L)\tensor_{R(\Lb)} \Mb(\Lb) \\
\lra \prod_{\pi (K)=\Kb} \cEi_{K} \prod_{\pi(L)=\Lb,
L\subseteq K}  e_{L}\Mb (\Lb)\\
=(\pes e\Mb)(\Kb\supseteq \Lb)
\end{multline*}

These together satisfy the triangular identities and give an
adjunction. 
\end{proof}

\subsection{$p$-modules}

The adjunction in Lemma \ref{lem:pshriekp} shows that the category of
qc-$\pi$-cts-$R^p$-modules is a retract of the category of 
i-qc-$R^p$-modules.

\begin{defn}
We say that a module over $\pes R^p$ is a $pqc$-module if
the natural map $\Mb \lra \pes e \Mb$  made explicit in Lemma
 \ref{lem:pshriekp} gives an isomorphism
$$\Mb (\Kb\supseteq \Lb)\cong \prod_{\pi (K)=\Kb} \cEi_{K} \prod_{\pi(L)=\Lb,
L\subseteq K}  e_{L}\Mb (\Lb). $$

For brevity, we write $pqc$ for $p$-modules which are $iqc$ and $pqce$ for
$p$-modules which are $iqc$ and $ie$ on the grounds that these are the
appropriate notions for $p$-modules. 
\end{defn}

\begin{remark}
(i) For length 0 flags, the condition states that the values on
are simply the products of the values of $e\Mb$ on the length 0 flags of
$\Sigma$:
$$\Mb(\Kb)=\prod_{\pi K=\Kb} e_K\Mb(\Kb). $$
This accounts for the letter $p$ in $pqc$.

(ii) We have already noted that $\pes e\Mb$ is a $qc$-module. 
Thus, the horizontal map $L \lra (K\supseteq L)$ induces
$$e_{K\supseteq L} \Mb(\Kb\supseteq \Lb)=\cEi_{K/L}e_L\Mb(\Lb). $$
This accounts for the letters $qc$ in $pqce$.
\end{remark}

\begin{defn}
We say that a $\pes R$-module $\Mb$ is is {\em i-extended} (or an
$ie$-module) if $e\Mb$ is extended as an $R$-module. 

More explicitly, the vertical inclusion $K\lra (K\supseteq L)$ induces
$\Mb(\Kb)\lra \Mb(\Kb\supseteq \Lb)$ and hence 
$R(\Kb\supseteq \Lb)\tensor_{R(\Kb)} \Mb(\Kb)\lra \Mb(\Kb\supseteq \Lb)$. If $\Mb$
is extended and we  choose $K\supset L$ with $\pi K=\Kb,
\pi L=\Lb$ then we obtain an isomorphism 
$$e_{K\supseteq L} \Mb(\Kb\supseteq \Lb)= R(K\supseteq L) \tensor_{R(K)} e_{K}\Mb(\Kb). $$
\end{defn}

\begin{cor}
\label{cor:RmodpesRmodp}
The adjunction of Lemma \ref{lem:pshriek}
gives an equivalence
$$\mbox{$pqc$-$\pes R^p$-modules} \simeq
\mbox{$qc$-$\pi$-cts-$R^p$-modules} , $$
and this restricts to an equivalence
$$\mbox{$pqce$-$\pes R^p$-modules} \simeq
\mbox{$qce$-$\pi$-cts-$R^p$-modules} .\qqed $$
\end{cor}

\section{$R^p$-modules and $(R\relax {\mathcal F})^p$-modules}
\label{sec:RRF}

For any surjective map $q: \Sigmat \lra \Sigma$, we have explained how
to compare $R^p$-modules and $q_!^eR^p$-modules. However when 
$q$ itself is an opfibration  there is the alternative of forming the splitting system $q_!R^s$, and comparing $R^p$-modules and
$(q_!R^s)^p$-modules.  We restrict attention to the case that 
$\Sigmat=\Sigma \cF$ is formed by introducing multiplicities, and
we have $q: \Sigma \cF \lra \Sigma$; accordingly we write $R\cF
=q_!R$. The case of flags follows from the case of pairs, so we will restrict to pairs. 

\begin{example}
If we take $\Sigma =\Sigma_c$ then $\Sigma \cF\cong \Sigma_a$ and $q:
\Sigma_a \lra \Sigma_c$. We then take $R^s=\RRa^s$, so that
$q_!R^s=\RRa \cF=\RRc^s$. Thus this section is exactly designed to consider the
relationship between $\RRa$-modules and $\RRc$-modules. 
\end{example}

One method is to observe that there is a map $\lambda : (R\cF)^p \lra
q_!^eR^p$  and then use restriction, extension and coextension of
scalars. Since the map $\lambda$ is an isomorphism on idempotent
pieces, this allows one to construct a right adjoint to $e:
\mbox{$(R\cF)^p$-modules}\lra \mbox{$R^p$-modules}$, namely
$\lambda^*q_!^e$. However this takes values in the 
$iqc$-modules  (i.e., ones which are $qc$ after $e$ is applied), which
are different from straightforwardly $qc$-modules and so
quasi-coherification would be necessary. 
Instead it seems better to work directly. 

Because of the special nature of $q: \Sigma \cF \lra \Sigma$ we may
formulate a stronger continuity condition on sections, requiring a
continuity condition on $\cF/K$ as well as on inclusions. We therefore
refer to this as $\cF$-continuity. 

\begin{defn}
An {\em $\cF$-$q$-structure} on  $qc$-$R^p$-module is a transitive
system of lifts  for all pairs $K \supseteq L$
$$\diagram 
&&\cEi_K\prod_{\tK}\prod_{\tL \subseteq \tK}\Mt(\tL)=\Mt(K\supseteq
L)_{\cF c}\dto \\
&&\prod_{\tK}\cEi_{\tK}\prod_{\tL \subseteq \tK}\Mt(\tL)=\Mt(K\supseteq L)_c\dto \\
\prod_{\tK}\Mt(\tK) \rrto \ar@{-->}[uurr] &&  \prod_{\tK}\prod_{\tL
  \subseteq  \tK}\cEi_{\tK/\tL}\Mt(\tL)=\Mt(K\supseteq L) 
\enddiagram$$
\end{defn}

\begin{defn}
\label{defn:qshriekd}
We define a functor 
$$q_!^d: \mbox{$q$-$\cF$cts-$qc$-$R^p$-modules}\lra \mbox{$qc$-$(R\cF)^p$-modules}$$
by 
$$(q_!^d\Mt)(K\supseteq L)=\cEi_K \prod_{\tL}\Mt(\tL) =\Mt(K\supseteq
L)_{\cF c}. $$
The horizontal structure maps are simply localizations, and the
vertical structure map for $v: (K)\lra (K\supseteq L)$ is the map
$$\prod_{\tK}\Mt(\tK) \lra \cEi_K\prod_{\tL\subseteq \tK} \Mt(\tL)$$
given by the $\cF$-$q$-structure. 
\end{defn}

\begin{remark}
The definition of an $\cF$-$q$-structure on $\Mt$ may now be rephrased
as saying that the values $\Mt(K\supseteq L)_{\cF c}$ fit together to make
an $(R\cF)^p$-module $q_!^d\Mt$, equipped with a map $q_!^de\Mt\lra
q_!e\Mt$, and an $\cF$-$q$-structure is a map $\Mt \lra q_!^de\Mt$.
\end{remark}

\begin{defn}
A $qc$-$(R\cF)^p$-module $M$ is a {\em $p$-module} if the unit $M\lra q_!^deM$
gives an isomorphism 
$$N (K\supseteq L)\stackrel{\cong}\lra \cEi_K\prod_{\tL} e_{\tL}M(L).$$
\end{defn}

\begin{lemma}
\label{lem:eqd}
There is an adjunction 
$$\adjunction{e}
{\mbox{$qc$-$(R\cF)^p$-modules}}
{\mbox{$q$-$\cF$cts-$qc$-$R^p$-modules}}
 {q_!^d}$$
which restricts to an equivalence
$$\mbox{$pqc$-$(R\cF)^p$-modules}\simeq
\mbox{$q$-$\cF$cts-$qc$-$R^p$-modules}. \qqed$$
\end{lemma}

The functor $e$ takes extended modules to extended modules, but
$q_!^d$ does not. Instead, if we compose with the associated extended
module functor $\Gamma_v$ from \cite{tnq2} we obtain the following. 

\begin{cor} 
\label{cor:egammaqd}
If $\Sigma$ is finite, 
there is an equivalence
$$\mbox{$qce$-$(R\cF)^p$-modules}\simeq \mbox{$q$-$\cF$cts-$qce$-$R^p$-modules}\qqed$$
\end{cor}

\section{Applications to models for rational torus-equivariant spectra}
\label{sec:AG}

The purpose of this section is to record the consequences of the
general theory for the special case relevant to rational $G$-spectra
where $G$ is an $r$-torus. In this subsection we introduce the
diagrams of rings and the modules over them, in Subsection \ref{subsec:diagrams} we will
display the categories and functors, and then in a series of
subsections we describe how our general results establish the
equivalences we need. 

We consider  $\Sigma_c=\connsub (G)$ ($c$ stands for `connected'), and the standard
system of multiplicities so that $\Sigma_a=\Sigma \cF =\TC (G)$ is the
toral chain category ($a$ stands for `all') and the dimension poset $\Sigma_d=\{ 0, 1, \ldots,
r\}$ ($d$ stands for `dimension').   We consider the $\Sigma_a$-splitting system $\RRa$ defined by
$$\RRa (G/K)=H^*(BG/K) $$ 
equipped with its standard system of Euler classes, which is maximally
generated. We could then introduce multiplicities and hence get a
$\Sigma_c$-splitting system $\RRcb =(i^*\RRa)\cF$, but as shown in
Example \ref{eg:rrcfisrrcbf},  this is isomorphic to  the $\Sigma_c$-system $\RRc$ defined
by 
$$\RRc (G/K)=\RRa \cF (G/K)= \prod_{\tK \in \cF
 /K} H^*(BG/\tK )=\cO_{\cF/K}. $$
Note that the associated coefficient system of this splitting system
is middle independent
so is quite different from $q_!^e\RRa$.

The other important coefficient system is $\RRd =d_!^e\RRc$. This is
defined on the subdivided $r$-simplex $\flag ([0,r])$, and takes the
following values on vertices:  
$$\RRd (m)=d_!^e\RRc (m) \cong \prod_{\dim (K )=m} \cO_{\cF/K} \cong
\prod_{\dim (\tK)=m }  H^*(BG/\tK ). $$

The purpose of this section is to assemble all the work we have done
to give adjoint equivalences between categories of $\RRd^f$-modules
(i.e., modules over a coefficient system on the subdivided
$r$-simplex, as comes out of the topology), categories of 
$\RRc^{p}$-modules (i.e., modules over pairs as in \cite{tnq1, tnq2})
and categories of  
$\RRa^{p}$-modules (i.e., modules over pairs encoding the
localization theorem).

More precisely, we show the following four versions of the category
$\cA(G)$ are equivalent:
\begin{itemize}
\item $\cA_a^p(G) :=\mbox{$q$-$\cF$cts-$qce$-$\RRa^p$-mod}$ (the model
  most clearly embodying the localization theorem)
\item $\cA_c^p(G) :=\mbox{$qce$-$\RRc^p$-mod}$ (the model used in
  previous work from \cite{tnq1} onwards)
\item $\cA_c^f(G):=\mbox{$qce$-$\RRc^f$-mod}$ (the model used to
  compare with $\cA_d^f(G)$ below)
\item $\cA_d^f(G):=\mbox{$pqce$-$\RRd^f$-mod}$ (the model coming out
  of the proof in \cite{tnqcore})
\end{itemize}

We note that there are numerous other variants that could be
discussed (for example (i) replace $\RRa^p$ by $\RRa^f$,
(ii) replace $\RRc^p$ by $q_!^d\RRa^p$, 
(iii) replace $\RRc^f$ by $q_!^e\RRa^f$, (iv) 
$\RRd^f$ by $(dq)_!^e\RRa^f$). Our general results do give models
based on each of these alternatives, but we will focus on the four listed.

\subsection{Some diagrams}
\label{subsec:diagrams}
In the following diagrams, a number of  functors are used, and their
definitions may be found as follows: $e$ in \ref{lem:defe}, 
$q_*$ in \ref{defn:pistar}, 
$q_!$ in \ref{defn:pishriek}, 
$q_!^d$ in \ref{defn:qshriekd}, 
$d_!^e$ in \ref{defn:pishrieke},
$p$ and $f$ in Section \ref{sec:pf}. The definitions of the categories
are recalled near those of the appropriate functors.

To start with, we display the module categories of concern, together with the
functors that exist on the whole module categories. 
$$\diagram
&a& c&d\\
s&\mbox{$\RRa^s$-modules}
\rto<-0.7ex>_{q_!}&\mbox{$\RRc^s$-modules} \lto<-0.7ex>_e& \times\\
p&\mbox{$\RRa^p$-modules}
\dto^f&\mbox{$\RRc^p$-modules} \lto<-0.7ex>_e\dto^f& \times\\
f&\mbox{$\RRa^f$-modules}
&\mbox{$\RRc^f$-modules} \lto<-0.7ex>_e 
& \mbox{$\RRd^f$-modules}\lto<-0.7ex>_e
\\
\enddiagram$$

Restricting to categories of 
$qc$-modules on which the Euler-adapted right adjoints exist, we have
the diagram
$$\diagram
&a& c&d\\
s&\mbox{$\RRa^s$-modules}
\rto<-0.7ex>_{q_!}&\mbox{$\RRc^s$-modules} \lto<-0.7ex>_e& \times\\
p&\mbox{$qc$-$q$-$\cF$cts-$\RRa^p$-modules}\rto<-0.7ex>_{q^d_!}
&\mbox{$qc$-$\RRc^p$-modules} \lto<-0.7ex>_e\dto<0.7ex>^f & \times\\
f&\times
&\mbox{$qc$-$\RRc^f$-modules} \rto<-0.7ex>_{d^e_!}\uto<0.7ex>^p
& \mbox{$qc$-$\RRd^f$-modules}\lto<-0.7ex>_e
\\
\enddiagram$$

Restricting further to categories on which we have equivalences, and
using the torsion functor $\Gamma $ right adjoint to inclusion of
$qce$-$\RRc^p$-modules into all $\RRc^p$-modules defined in \cite{tnq2} (see also Section
\ref{sec:Gamma} below)
$$\diagram
&a& c&d\\
s&\times &\times &\times\\
p&\mbox{$qce$-$q$-$\cF$cts-$\RRa^p$-modules}\rto<-0.7ex>_{\Gamma q^d_!}
 &\mbox{$qce$-$\RRc^p$-modules} \dto<0.7ex>^f \lto<-0.7ex>_e^{\simeq}& \times\\
f&\times
&\mbox{$qce$-$\RRc^f$-modules}
\rto<-0.7ex>_{d^e_!}^{\simeq}\uto<0.7ex>^p_{\hspace{-0.7ex} \simeq}
& \mbox{$pqce$-$\RRd^f$-modules}\lto<-0.7ex>_e
\\
\enddiagram$$

More succinctly
$$\diagram
&a& c&d\\
s&\times&\times &\times \\
p&\cA_a^p(G) \rto<-0.7ex>_{\Gamma q^d_!}^{\simeq}
&\cA_{c}^p(G) \dto<0.7ex>^f \lto<-0.7ex>_e& \times\\
f&\times
&\cA_{c}^f (G) \rto<-0.7ex>_{d^e_!}^{\simeq}\uto<0.7ex>^p_{\hspace{-0.7ex}\simeq}
& \cA^f_{d}(G)\lto<-0.7ex>_e
\\
\enddiagram$$

\subsection{Flags and pairs}
For the two posets $\Sigma_a$ and $\Sigma_c$, we have splitting systems and we may therefore define
categories of  pairs. For each of these Lemma
\ref{lem:QPisflag} gives 
 an equivalence between a pair of algebraic models of rational $G$-spectra. 

\begin{cor}
(i) There is an equivalence of categories
$$\cAp_a(G)=\mbox{$qce$-$q$-$\cF$cts-$\RRa^p$-modules}\simeq 
\mbox{$qce$-$q$-$\cF$cts-$\RRa^f$-modules}=\cA^{f}_a(G) $$ 

(ii) There is an equivalence of categories 
$$\cA^{p}_c(G)=\mbox{$qce$-$\RRc^p$-modules}\simeq
\mbox{$qce$-$\RRc^f$-modules}=\cAf_c(G) . \qqed$$ 
\end{cor}
\subsection{$\RRa^p$-modules and $\RRc^p$-modules}

We apply the results of Section \ref{sec:RRF} to compare
$\RRa^p$-modules and $\RRc^p$-modules. Special cases of Lemma
\ref{lem:eqd} and Corollary \ref{cor:egammaqd} give the following. 

\begin{lemma}
There is an adjunction 
$$\adjunction{e}
{\mbox{$qc$-$\RRc^p$-modules}}
{\mbox{$q$-$\cF$cts-$qc$-$\RRa^p$-modules}}
{q_!^d}$$
which restricts to an equivalence
$$\mbox{$pqc$-$\RRc^p$-modules}\simeq \mbox{$q$-$\cF$cts-$qc$-$\RRa^p$-modules}$$

If we compose with the  functor $\Gamma$ from \cite{tnq2} we obtain an equivalence
$$\cA_c^p(G)=\mbox{$qce$-$\RRc^p$-modules}\simeq \mbox{$q$-$\cF$cts-$qce$-$\RRa^p$-modules}=\cA_a^p(G).\qqed$$
\end{lemma}

\subsection{Collecting subgroups of the same dimension}

In this subsection we consider the dimension function $d: \Sigma_c\lra \Sigma_d=[0,r]$.
Since $\Sigma_c$ has a minimal element, by Lemma \ref{lem:mininducespstructure} we do not need to mention
$d$-continuity, and we prove that  the category of $qce$-modules over the flag
complex of all connected subgroups (i.e., $qce$-$\RRc^f$-modules) 
is  equivalent to the category of pqce-modules over the subdivided $r$-simplex
(i.e., $pqce$-$\RRd^f$-modules).  
Corollary \ref{cor:RmodpesRmod} has the following special case.

\begin{cor} 
With 
$$\RRc (G/K)=\prod_{\tK \in \cF/K} H^*(BG/\tK) =\cO_{\cF/K} \mbox{ and } \RRd (m)=\prod_{\dim (K)=m}\cO_{\cF/K}$$
as above, there is an adjunction 
$$
\adjunction{e}{\mbox{$qc$-$\RRd^f$-mod}}
{\mbox{$qc$-$\RRc^f$-mod}}{d^e_!}
$$
This induces an equivalence
$$\mbox{$qc$-$\RRc^f$-mod}\simeq \mbox{$pqc$-$\RRd^f$-mod}. $$
Furthermore, it  respects extended modules, and induces an equivalence
$$\cAf_{c}(G)=\mbox{$qce$-$\RRc^f$-mod}\simeq
\mbox{$pqce$-$\RRd^f$-mod}=\cAf_d(G). \qqed$$
\end{cor}

Note that composing our equivalences does not give a simple comparison of $\RRa^f$-modules
and $\RRd^f$-modules. The continuity conditions  mean that
$(dq)_!^e\RRa^f\not \cong \RRd^f$ (even though they
have the same values on vertices). Instead we have the following
equivalence.

\begin{cor} 
We have an equivalence
$$\mbox{$dq$-cts-$qc$-$\RRa^f$-mod}\simeq \mbox{$pqc$-$((dq)_!^e\RRa^f)$-mod}. $$
Furthermore, it  respects extended modules, and induces an equivalence
$$\mbox{$qce$-$dq$-cts-$\RRa^f$-mod}\simeq
\mbox{$ppqce$-$(dq)_!^e\RRa^f$-mod}. \qqed$$
\end{cor}

\begin{remark}
We could presumably also obtain this equivalence as a composite. 
\end{remark}

\subsection{The case of rank 1}
The case of the circle group is too simple to be representative. To
start with $\Sigma_c=(1\lra T)$ is finite, avoiding one layer
of continuity conditions. The fact that the chains are of length $\leq
1$ means there is no distinction between pairs (p) and flags
(f). Finally, since all subgroups of dimension 0 have the same
identity component $\Sigma_d=\Sigma_c$. In short, the only two models
that really need comparision are $\cA^p_c(T)$ and $\cA^p_a(T)$. This
has been discussed in \cite{s1q}, but it is helpful to relate it to
the present framework and notation. The recollections here will also
be useful preparation for Section \ref{sec:Gamma}. 

In any case 
$$\Sigma_a= \{ i\lra T \st i\geq 1\}, $$
where $i$ is short for the cyclic subgroup of order $i$. The map
$q:\Sigma_a\lra \Sigma_c$ is defined by $q(i)=1$ and $q(T)=T$.
The diagram of rings is specified by the values on $i\lra T$, which we
abbreviate to  
$$R_i \lla k$$
In the original case 
$$R_i=H^*(B(T/C_i))=\Q [c_i] \mbox{ and } k=\Q ,$$
where $c_i$ is of cohomological degree 2. We then find that on the diagram $\Sigma_c$ we have
$$ R=\prod_i R_i \lla k. $$

Next, 
$$\flag (\Sigma_c)=\left\{ (1)\lra (T\supset 1)\lla (T) \right\}$$
and 
$$\flag (\Sigma_a)=\left\{ (i)\lra (T\supset i)\lla (T) \st i\geq 1
\right\}. $$
The system of Euler classes is given by choosing an element $c_i\in
R_i$. The value on the $i$th flag from $\flag(\Sigma_a)$ is
$$R_i \lra R_i[\frac{1}{c_i}] \lla k$$
and the value on the flag from $\flag(\Sigma_c)$ is 
$$R \lra \cEi R  \lla k$$
where 
$$\cEi R=\colim (R\stackrel{(c_1,1,1,1,1\cdots)}\lra
  R\stackrel{(c_1,c_2,1,1,1\cdots)}\lra 
    R\stackrel{(c_1,c_2,1c_3,1,1\cdots)}\lra
    R\stackrel{(c_1,c_2,1c_3,c_4,1\cdots)}\lra \cdots ). $$
Now $\cA^p_c(T)$ is a certain category of $\R_c$-modules, namely
$$\cA^p_c(T)=\{ N\lra P \lla V\st \cEi N\cong P \cong \cEi R\tensor_kV\}$$
(where the first isomorphism is quasicoherence, and the second is
extendedness). 

Next,  $\cA^p_a (T)$ is a certain category of $\R_a$-modules, namely
$$\cA^p_a(T)=\{ N_i\lra P_i \lla V\st  N_i[\frac{1}{c_i}]\cong P_i
\cong R[\frac{1}{c_i}]\tensor_kV,\\ 
\diagram 
&\cEi \prod_i N_i\dto \\
V\urto^{\kappa} \rto &\prod_i (N_i[\frac{1}{c_i}])
\enddiagram
\}$$
(where the first isomorphism is quasicoherence, the second
is extendedness, and $\kappa$ is the continuity structure). 

Finally, the equivalence between these two categories is induced by
$e$ and $\Gamma q_!^d$. In the easier direction
$$e: \cA^p_c(T)\lra \cA^p_a(T)$$
takes $N\lra P \lla V$ to the object with $N_i=e_iN, P_i=e_iP$ and
the continuity condition $\kappa$ is the composite
$$V\lra P\cong \cEi N\lra \cEi \prod_ie_iN. $$
In the other direction, 
$$\Gamma q_!^d: \cA^p_a(T)\lra \cA^p_c(T)$$
takes an object $\{ N_i\lra P_i \lla V, \kappa\}$ first by $q_!^d$ to  
$$(\prod_i N_i \lra \prod_i P_i \lla V)$$
and then by $\Gamma$ to 
$(N \lra P \lla V)$ where $N$ is the pullback
$$\diagram
N \rto \dto & \cEi R\tensor_k V\dto \\
\prod_iN_i \rto &\cEi \prod_iN_i
\enddiagram$$
and $P=\cEi N$. 

The fact that these are inverse equivalences follows from the fact
that the square for $N=R$  is a pullback (i.e., that $\R_c=\Gamma
q_!^d e\R_c$),  together with the quasicoherence condition.

\section{Torsion functors}
\label{sec:Gamma}

Various obvious constructions on qce-modules give objects which are
not $qce$. It is therefore convenient to have a right adjoint $\Gamma$
to the inclusion of $qce$-modules in all modules. There are three cases
where  we need this: for $\RRc^p$-modules,  $\RRc^f$-modules and
$\RRd^f$-modules. The case of $\RRc^p$ was dealt with in \cite{tnq2}. 
It does not seem  possible to deduce the case of $\RRc^f$ directly, because one can
only  construct pair modules from flag modules in the
middle-independent case. Nonetheless, the methods of \cite{tnq2}
remain  effective. The case of $\RRd^f$ is a little different, because
the coproduct of $p$-modules is not the coproduct of the underlying
modules; this will be discussed in Subsection \ref{subsec:Gammad}. 

The strategy for the two cases involving $\RRc$ is the same for pairs
and  flags. As in \cite{tnq2}, we factorize the inclusion 
$$\mbox{$\RRc$-modules}\stackrel{k}\lra \mbox{$e$-$\RRc$-modules}\stackrel{j}\lra
\mbox{$qce$-$\RRc$-modules}.$$
The right adjoint to the first is called $\Gamma_v$ (the associated
extended module construction) and to the second $\Gamma_h$
(quasi-coherification); the letters $v$ and $h$ refer to horizontal
and vertical structure maps. 

The functor $\Gamma_v$ does not use very much about our particular
context: it would apply to any poset $\Sigma$ with ranks (i.e., with a
dimension function $d: \Sigma \lra [0,r]$ for some $r$, so that
$K\supset L$ implies $d(K)>d(L)$). However the construction of
$\Gamma_h$ needs finiteness conditions on the rings as well, and we
will be content to cover our immediate applications.

\begin{thm}
\label{thm:torsionfunctors}
There are right adjoints to inclusions as follows
\begin{enumerate}
\item $\Gamma^p_{cv}$ to 
$$k: \mbox{$\RRc^p$-modules} \lra \mbox{$e$-$\RRc^p$-modules}$$
\item $\Gamma^p_{ch}$ to 
$$j: \mbox{$e$-$\RRc^p$-modules} \lra \mbox{$qce$-$\RRc^p$-modules}$$
\item $\Gamma^p_c=\Gamma^p_{ch}\Gamma^p_{cv}$ to 
$$i=jk: \mbox{$\RRc^p$-modules} \lra \mbox{$qce$-$\RRc^p$-modules}$$
\item $\Gamma^f_{cv}$ to 
$$k: \mbox{$\RRc^f$-modules} \lra \mbox{$e$-$\RRc^f$-modules}$$
\item $\Gamma^f_{ch}$ to 
$$j: \mbox{$e$-$\RRc^f$-modules} \lra \mbox{$qce$-$\RRc^f$-modules}$$
\item $\Gamma^f_c=\Gamma^f_{ch}\Gamma^f_{cv}$ to 
$$i=jk: \mbox{$\RRc^f$-modules} \lra \mbox{$qce$-$\RRc^f$-modules}$$
\end{enumerate}
\end{thm}

We note that Part 1 is \cite[Theorem 7.1]{tnq2} and  Part 2 is
\cite[Theorem 8.1]{tnq2}. Part 3
follows from Parts 1 and 2 and Part 6 follows from Parts 4 and 5. 

We will prove Part 4 in Subsection \ref{subsec:Gammav}. 
In Subsection \ref{subsec:Gammah} we show how to deduce Part 5 from
Part 2. Finally in Subsection \ref{subsec:Gammad} we will discuss the
category of $\RRd^f$ modules and construct a right adjoint from
$\RRd^f$-modules to $qce$-$\RRc^p$-modules. 

\subsection{The associated extended functor}
\label{subsec:Gammav}

The purpose of this section is to give a construction of a functor
$\Gamma_v$ replacing an $\RRc^f$-module by an extended $\RRc^p$-module, so that
its vertical structure maps become extensions of scalars. The proof is
a direct adaption of \cite[Theorem 7.1]{tnq2}. In fact it applies
whenever $\Sigma$ has a dimension function as described above.

\begin{thm} 
There is a right adjoint $\Gamma_v=k^!$  to the inclusion
$$\mbox{e-$\RRc^f$-mod} \stackrel{k}\lra \mbox{$\RRc^f$-mod}.  $$
\end{thm}

We will give an explicit construction of the functor $k^!$, referring
the reader to \cite[Section 7]{tnq2} for motivating discussion.

\begin{defn}
(i) The {\em codimension} of a flag $F$ is the codimension of
the largest element, $fF$.

(ii)  
Given an $\RRc^f$-module  $M$, we describe the construction  of the 
associated extended module $k^!M$. We will describe how to define $k^!M$ on
singleton flags, $(k^!M)(L)$ and then use
extendedness to determine the values on other flags with first element
$L$: 
$$(k^!M)(E) :=R(E)\tensor_{R(L)} (k^!M)(L). $$

We proceed in order of increasing codimension, starting in codimension
0 by taking $(k^!M) (G)=M(G)$.  Assume $k^!M$ has been defined  on flags of codimension $\leq n$ in such a way that the vertical
maps for flags of codimension $\leq n$ from each point are extensions
of scalars. Now suppose $L$ is of codimension $n+1$. 

The value at $L$ is determined as an inverse limit of a diagram with
two rows, the zeroth given by existing values of $k^!M$  on flags of
codimension $n$ and the first by  the values of $M$ itself. The diagram takes the form
$$
\diagram 
\bullet \rto \dto &(K\supset L),0) \dto\\
((L),1) \rto &((K\supset L), 1)
\enddiagram$$
where $K$ runs through the codimension $\leq n$ subgroups containing
$L$. Since we are defining a middle-independent module, it is not
necessary to mention longer flags. More precisely, 
$$
k^!(L)=\ilim \left( \diagram 
&(k^!M)(K\supset L) \dto\\
M(L) \rto &M(K\supset L)
\enddiagram\right) $$
\end{defn}

\begin{lemma}
The maps $\lambda: k^!M \lra M$ induce isomorphisms
$$\lambda_*:\Hom (T,k^!M) \lra \Hom (T,M )$$
for any extended $\RRc^f$-module $T$. In particular $k^!$ is right adjoint
to the inclusion 
$$k: \mbox{e-$\RRc^f$-mod}\lra \mbox{$\RRc^f$-mod}.$$
\end{lemma}
\begin{proof}
To see $\lambda_*$ is an epimorphism, suppose $f: T\lra M$ is a map, and 
we attempt to lift it to a map $f': T\lra k^!M$. We start
with $f'(G)=f(G)$, and proceed by induction on the codimension of the
flag. Once we have defined on a singleton flag $f'(L)$ we are forced
to use extendedness to define $f'(E)=\RRc (E)\tensor_{\RRc (L)}f'(L)$ on
flags $E$ with first term $L$. 

Now suppose $L$ is of codimension $n+1$ and that $f'$ is defined on
flags of lower codimension. 
$$\diagram T(L) \rto \dto
&(k^!M)(K\supset L) \dto\\
M(L) \rto &M(K\supset L)
\enddiagram $$
Since $(k^!M) (L)$ is defined as an inverse limit, we use its universal
property to give $f'(L): T(L)\lra k^!M(L)$. 

To see $\lambda_*$ is a monomorphism, suppose the two maps $f_1,f_2:L
\lra M$ give the same map to $k^!M$.  Evidently $f_1(G)=f_2(G)$, so we
may consider the minimum codimension in which they disagree, perhaps
codimension $n+1$; by extendedness they must therefore disagree on a
subgroup $L$ of codimension $n+1$. However the defining diagram for
$k^!M(L)$ involves only the values in $M$ and in flags of lower
codimension so the universal property shows $f_1(L)=f_2(L)$,
contradicting the choice of $L$. 
\end{proof}

\subsection{The horizontal torsion functor $\Gamma_h$}
\label{subsec:Gammah}

Consider the diagram
$$
\diagram
\mbox{$qce$-$\RRc^p$-modules} \rto<0.7ex>^f \dto<-0.7ex>_j &
\mbox{$qce$-$\RRc^f$-modules} \dto<-0.7ex>_j \lto<0.7ex>^p\\
\mbox{$e$-$\RRc^p$-modules} \rto<0.7ex>^f \dto<-0.7ex>_k \uto<-0.7ex>_{\Gamma^p_{ch}} &
\mbox{$e$-$\RRc^f$-modules} \dto<-0.7ex>_k \lto<0.7ex>^p \uto<-0.7ex>_{\Gamma^f_{ch}}\\
\mbox{$\RRc^p$-modules} \rto<0.7ex>^f  \uto<-0.7ex>_{\Gamma^p_{cv}}&
\mbox{$\RRc^f$-modules} \uto<-0.7ex>_{\Gamma^f_{cv}}\\
\enddiagram$$

The two $\Gamma^p$ functors on the left were constructed in
\cite{tnq2}. The functor $\Gamma^f_{cv}$ at the bottom right was
constructed in Subsection \ref{subsec:Gammav}. The right adjoint to
$j^f=fj^pp$ is  the composite  $\Gamma^f_{cv}:=f\Gamma^p_{ch}p$. The
point is that $f$ and $p$ are quasi-inverse and hence both left and
right adjoint to each other.

\subsection{The dimensional torsion functor $\Gamma_d$}
\label{subsec:Gammad} A special case of the results of Section
\ref{sec:pistar}  shows that the functor $e$ from $\RRd^f$-modules to
$\RRc^f$-modules has a {\em left} adjoint, $d_*$. This gives a diagram
$$
\diagram
\mbox{$qce$-$\RRc^f$-modules} \rto<-0.7ex>_{d_!^e}\dto<-0.7ex>_i &
\mbox{$pqce$-$\RRd^f$-modules} \dto<-0.7ex>_{d_*ie}\lto<-0.7ex>_e\\
\mbox{$\RRc^f$-modules}  \uto<-0.7ex>_{\Gamma^f_{c}} \rto<0.7ex>^{d_*}&
\mbox{$\RRd^f$-modules}\lto<0.7ex>^e \uto<-0.7ex>_{\Gamma^f_{d}}\\
\enddiagram$$
in which $\Gamma^f_d=d_!^e\Gamma_c^f e$. This is right adjoint to
$d_*ie$, and not to $i$ (typically $d_*ieM$ will not be a
$p$-module). This reflects the fact that $i$ is not a left adjoint
(coproducts in the category of $p$-modules are not coproducts in the
ambient category of $\RRd^f$-modules). In our applications \cite{tnqcore} we will
actually use $p\Gamma_c^fe: \mbox{$\RRd^f$-modules}\lra
\mbox{$qce$-$\RRc^p$-modules}$, which in any case takes us to the
category $\cA_c^p(G)$ that we want to work with.

\end{document}